\documentclass{amsart}
\usepackage{graphicx}
\usepackage{graphics}
\usepackage{psfrag}
\usepackage[active]{srcltx}
\usepackage{amssymb}
\usepackage{amscd}

\numberwithin{equation}{section}

\theoremstyle{plain}
\newtheorem{theorem}{Theorem}[section]
\newtheorem{lemma}[theorem]{Lemma}

\newtheorem{corollary}[theorem]{Corollary}

\theoremstyle{definition}
\newtheorem{definition}[theorem]{Definition}

\newtheorem{con}[theorem]{Convention}

\newcommand\varnot{\varnothing}
\newcommand\mo{^{-1}}
\newcommand\conv{^{\scriptstyle\smallsmile}}
\newcommand\ra{relation algebra}
\newcommand\g{\gamma}
\newcommand\lph{\alpha}
\newcommand\ka{\kappa}
\newcommand\seq{\subseteq}
\newcommand\vph{\varphi}
\newcommand\wi{xy}
\newcommand\wj{yz}
\newcommand\gxhi{G_\wx/H_\xy}
\newcommand\gyki{G_\wy/K_\xy}
\newcommand\xy{{xy}}
\newcommand\gm{\gamma}
\newcommand\tbigcup{\textstyle \bigcup}
\newcommand\vp{\varphi}
\newcommand\wx{x}
\newcommand\wy{y}
\newcommand\kp{\kappa}
\newcommand\qeddef{\qed}
\newcommand\dlt{\delta}
\newcommand\bt{\beta}
\newcommand\tbigcap{\textstyle \bigcap}
\newcommand\setdiff{\sim}
\newcommand\setcomp{\sim}
\newcommand\opar{\textnormal{(}}     
\newcommand\cpar{\textnormal{)}}     
\newcommand\xx{{xx}}
\newcommand\ex[1]{e_{#1}}
\newcommand\la{\lambda}
\newcommand\vth{\vartheta}
\newcommand\gp{G/P}
\newcommand\tsbigc{\textstyle \bigcup\limits}
\newcommand\yx{{yx}}
\newcommand\yz{{yz}}
\newcommand\x{\xi}
\newcommand\Ga{\varGamma}
\newcommand\wwz{{wz}}
\newcommand\xz{{xz}}
\newcommand\M{M}
\newcommand\gxm{G_\wx/\m 0}
\newcommand\gyp{G_\wy/\p 0}
\newcommand\gzn{G_\wz/\n 0}
\newcommand\De{\varDelta}
\newcommand\wz{z}
\newcommand\ch{\chi}
\newcommand\zt{\zeta}
\newcommand\rh{\rho}
\newcommand\sgm{\sigma}
\newcommand\relprod{\bigm\vert}
\newcommand\Ps{\varPsi}
\newcommand\Ph{\varPhi}
\newcommand\kh{K_{\xy}\scir\h\yz}
\newcommand\hh{\h\xy\scir\h\xz}
\newcommand\zx{{zx}}
\newcommand\zy{{zy}}
\newcommand\ggra{group relation algebra}
\newcommand\kk{K_{xz}\scir K_{yz}}
\newcommand\yy{{yy}}
\newcommand\zz{{zz}}
\newcommand\mce{\mathcal{E}}
\newcommand\scir{\mathbin{\raise2pt\hbox{$\scriptscriptstyle\circ$}}}
\newcommand\rp{\mid}

\newcommand\gcm[1]{{\mathfrak {Cm}}(#1\/)}
\newcommand\rr[1]{R_{#1}}
\newcommand\idd[1]{id_{#1}}
\newcommand\h[1]{H_{#1}}
\newcommand\diverse[1]{di_{#1}}
\newcommand\full[1]{\mathfrak{Re}({#1})}
\newcommand\cs[2]{{#1}_{#2}}
\newcommand\f[1]{{\mathfrak {#1}}}
\newcommand\id[1]{id_{#1}}
\newcommand\pair[2]{(#1,#2)}
\newcommand\bbigcup[1]{{\textstyle\bigcup_{#1}}}
\newcommand\com[1]{\mathfrak{Cm}({#1})}
\newcommand\G[1]{G_{#1}}
\newcommand\e[1]{e_{#1}}
\newcommand{\mc}[1]{\mathcal{#1}}
\newcommand\vphi[1]{\vph_{#1}}
\newcommand\hs[2]{H_{#1,#2}}
\newcommand\ks[2]{K_{#1,#2}}
\newcommand\kai[1]{\ka_{#1}}
\newcommand\gsq[2]{\G{#1}\times \G {#2}}
\newcommand\ho[2]{\varphi_{#1#2}}
\newcommand\vphih[1]{{\hat\vph}_{#1}}
\newcommand\p[1]{P_{#1}}
\newcommand\m[1]{M_{#1}}
\newcommand\n[1]{N_{#1}}
\newcommand\subs[2]{#1_{#2}}
\newcommand\trip[3]{(#1,#2,#3)}
\newcommand\kap[2]{\kappa_{{#1}{#2}}}
\newcommand\newrr[3]{R_{{#1}{#2},{#3}}}
\newcommand\newhl[2]{H_{{#1}{#2}}}
\newcommand\newhr[2]{K_{{#1}{#2}}}
\newcommand\newgl[2]{G_{#1}/H_{{#1}{#2}}}
\newcommand\newgr[2]{G_{#2}/K_{{#1}{#2}}}
\newcommand\grp[1]{G_{#1}}
\newcommand\cra[2]{\mathfrak{#1}{[{\mathcal{#2}}]}}
\newcommand\craset[2]{#1{[{\mathcal{#2}}]}}
\newcommand\cras[2]{\mathfrak{#1}{[{#2}]}}
\newcommand{\SB}[1]{\ensuremath{\textit{Sb}{({#1}\/)}}}

\newcommand\rrcomma{\,\textnormal{,}}
\newcommand\smbcomma{\,,}
\newcommand{\comma}{\textnormal{,}\ }
\newcommand\per{\textnormal{\myspace.\ }}
\newcommand\ident{1\mynegspace\textnormal{\rq}}
\newcommand\co{\textnormal{,}\ }     
\newcommand\po{\textnormal{.}\ }     
\newcommand{\myspace}{{\hspace*{.5pt}}} 
\newcommand\mmmyspace{{\hspace*{.1pt}}}  
\newcommand{\mynegspace}{{\hspace*{-.5pt}}}
\newcommand{\myshortspace}{{\hspace*{.1pt}}}

\begin{document}


\title[Relation algebras and groups]{Relation algebras and groups}

\author[S. Givant]{Steven Givant}
\address{Mills College\\5000 MacArthur Boulevard\\Oakland CA 94613\\USA}
\email{givant@mills.edu}

\thanks{This research was partially supported  by Mills College.}

\dedicatory{This article is dedicated to Bjarni J\'onsson}



\begin{abstract}
Generalizing results of J\'onsson and Tarski, Maddux introduced the
notion of a \textit{pair-dense} relation algebra and proved that
every pair-dense relation algebra is representable. The notion of a
pair below the identity element is readily definable within the
equational framework of relation algebras.  The notion of a triple,
a quadruple, or more generally, an element of size (or measure)
$n>2$ is not definable within this framework, and therefore it seems
at first glance that Maddux's theorem cannot be generalized. It
turns out, however, that a very far-reaching generalization of
Maddux's result is possible if one is willing to go outside of the
equational framework of relation algebras, and work instead within
the  framework of the first-order theory. Moreover, this
generalization sheds a great deal of light not only on Maddux's
theorem, but on the earlier results of J\'onsson and Tarski.

In the present paper, we define the notion of an atom below the
identity element in a relation algebra having measure $n$ for an
arbitrary cardinal number $n>0$, and we define a relation algebra to
be \textit{measurable} if it's identity element is the sum of atoms
each of which has some (finite or infinite) measure. The main
purpose of the present paper is to construct a large class of new
examples of \textit{group relation algebras} using systems of groups
and corresponding systems of quotient isomorphisms (instead of the
classic example of using a single group and forming its complex
algebra), and to prove that each of these algebras is an example of
a measurable set relation algebra. In a subsequent paper, the class
of examples will be greatly expanded by adding a third ingredient to
the mix, namely systems of ``shifting" cosets. The expanded class of
examples---called \textit{coset relation algebras}---will be large
enough to prove a representation theorem saying that every atomic,
measurable relation algebra is essentially isomorphic to a coset
relation algebra.
\end{abstract}

\maketitle


\section{Introduction}\label{S:1}

The calculus of  relations was created by De\,Morgan\,\cite{dm},
 Peirce (see, for example, \cite{pe}), and Schr\"oder\,\cite{sc}
in the second half of the nineteenth century. It was intended as an
algebraic theory of binary relations analogous in spirit to Boole's
algebraic theory of classes, and much of the early work in the
theory consisted of a clarification of some of the important
operations on and to binary relations and a study of the laws that
hold for these operations on binary relations.

It was Peirce\,\cite{pe} who ultimately determined the list of
fundamental operations, namely the Boolean operations on and between
binary relations (on a \textit{base set} $U$) of forming (binary)
unions, intersections, and (unary) complements (with respect to the
universal binary relation $U\times U$); and relative operations of
forming the (binary) relational composition---or relative
product---of two relations $R$ and $S$ (a version of functional
composition),
\[R\mathbin{\vert} S=\{\pair \alpha\beta: \pair \alpha\gamma\in R\text{ and
}\pair\gamma\beta\in S\text{ for some $\gamma$ in $U$}\};
\]  a
dual (binary) operation of relational addition---or forming the
relative sum---of $R$ and $S$,
\[R\mathbin{\dag} S=\{\pair \alpha\beta: \pair \alpha\gamma\in R\text{ or
}\pair\gamma\beta\in S\text{ for all $\gamma$ in $U$}\};
\]
and a unary op\-er\-a\-tion of rela\-tion\-al in\-verse (a ver\-sion
of func\-tion\-al in\-ver\-sion)---or forming the converse---of $R$,
\[R\mo=\{\pair \beta\alpha:\pair\alpha\beta\in R\}\per
\]  He also specified some distinguished relations on the set $U$: the empty
relation $\varnot$, the universal relation $U\times U$, the identity
relation
\[\id
U=\{\pair\alpha\alpha:\alpha\in U\},\] and its complement the
diversity relation
\[\diverse
U=\{\pair\alpha\beta:\alpha,\beta\in U\text{ and }\alpha\neq
\beta\}\per\]

Tarski, starting with\,\cite{t41}, gave an abstract algebraic
formulation of the theory.  As several of Peirce's operations are
definable in terms of the remaining ones, he reduced the number of
primitive operations to the Boolean operations of addition $\,+\,$
and complement $\,-\,$, and the relative operations of relative
multiplication $\,;\,$ and converse $\,\conv\,$, with an identity
element $\ident$ as the unique distinguished constant.  Thus, the
models for his set of axioms are algebras of the form
\[\f A=( A\smbcomma +\smbcomma
-\smbcomma ;\smbcomma\,\conv\smbcomma\ident)\comma  \]  where $A$ is
a non-empty set called the \textit{universe} of $\f A$, while
$\,+\,$ and $\,;\,$ are binary operations called \textit{addition}
and \textit{relative multiplication},
 $\,-\,$ and $\,\conv\,$ are unary operations called \textit{complement}
and \textit{converse}, and $\ident$ is a distinguished constant
called the \textit{identity element}.  He defined a relation algebra
to be any algebra of this form in which a set of ten equational
axioms is true. These ten axioms are true in any set relation
algebra, and the set-theoretic versions of three of them    play a
small role in this paper, namely the associative law for relational
composition, and first and second involution laws for relational
converse:\[ R\rp(S\rp T)=(R\rp S)\rp T,\qquad
(R\mo)\mo=R\comma\qquad (R\rp S)\mo=S\mo\rp R\mo \per\]

Tarski raised the problem whether all relation algebras---all models
of his axioms---are representable in the sense that they are
isomorphic to   set relation algebras, that is to say, they are
isomorphic to subalgebras of (full) set relation algebras
\[\full E=( \SB E\smbcomma \cup\smbcomma \sim\smbcomma
\,\rp\,\smbcomma\,{}\mo\smbcomma\id U)\] in which the universe $\SB
E$ consists of all subrelations of some equivalence relation $E$ on
a base set $U$, and the operations are the standard set-theoretic
ones defined above, except that complements are formed with respect
to $E$ (which may or may not be the universal relation $U\times U$).
Tarski and J\'onsson\,\cite{jt52} proved several positive
representation theorems  for classes of relation algebras with
special properties.  However, a negative solution to the general
problem was ultimately given by Lyndon\,\cite{lyn1}, who constructed
an example of a finite relation algebra that possesses no
representation at all. Since then, quite a number of papers have
appeared in which representation theorems for various special
classes of relation algebras have been established, or else new
examples of non-representable relation algebras have been
constructed. In particular, Maddux\,\cite{ma91}, generalizing
earlier theorems of J\'onsson-Tarski\,\cite{jt52}, defined the
notion of a pair-dense relation algebra, and proved that every
pair-dense relation algebra---every relation algebra in which the
identity element is a sum of ``pairs", or what might be called
singleton and doubleton elements---is representable.

In trying to generalize Maddux's theorem, a problem arises. The
property of being a pair below the identity element $\ident$ is
naturally expressible in the equational language of relation
algebras. Generally speaking, however, the size of an element below
the identity element---even for small sizes like $3$,  $4$ or
$5$---is \textit{not} expressible equationally.  To overcome this
difficulty  another way must be found of expressing  size, using the
first-order language of relation algebras. This leads to the notion
of a measurable atom.

 For an element $x$ below  the identity element---a
\textit{subidentity element}---the \textit{square} on $x$ (or the
\textit{square with side} $x$) is defined to be the element $x;1;x$.
In set relation algebras with unit $E=U\times U$, such squares are
just Cartesian squares, that is to say, they are relations of the
form $X\times X$ for some subset $ X$ of the base set $U$.  A
\textit{subidentity atom} $x$ is said to be \textit{measurable} if
its square $x;1;x$ is the sum (or the supremum) of a set of non-zero
functional elements, and the number of non-zero functional elements
in this set is called the \textit{measure}, or the \textit{size}, of
the atom $x$. If the set is finite, then the atom is said to have
\textit{finite measure}, or to be \textit{finitely measurable}. The
name comes from the fact that, for set relation algebras in which
the unit $E$ is the universal relation $U\times U$, the number of
non-zero
 functional elements beneath the square on a subidentity atom is
 precisely the same as the number of pairs of elements that belong to that atom. For
 instance, in such an algebra, a subidentity atom consists of a
 single ordered pair just in case its square is a function; it
 consists of two ordered pairs just in case its square is the sum
 of two non-empty functions; it consists of three ordered pairs
 just in case its square is the sum of three non-empty functions;
 and so on.

 In fact, the atoms below the square $x;1;x$ of a measurable
 subidentity atom $x$ may be thought of as ``permutations" of $x$,
 and they form a group $\cs Gx$ under the restricted operations of
 relative multiplication and  converse, with $x$ as the identity
 element of the group. Moreover, the set of atoms below an
 arbitrary rectangle $x;1;y$ (with $x$ and $y$ measurable atoms)
 also form a group, one that is isomorphic to a quotient of $\cs
 Gx$.

A relation algebra  is defined to be \textit{measurable} if its
identity element is the sum of a set of measurable atoms.  If each
of the atoms in this set is in fact finitely measurable, then the
algebra is said to be \textit{finitely measurable}.   The pair-dense
relation algebras of Maddux are finitely measurable, and in fact
each subidentity atom has measure one  or two.  The
 purpose of this paper and  \cite{ag} is to construct two classes of  measurable
  relation algebras: the class of group relation algebras, which is constructed
in this paper; and the broader class of coset relation algebras,
which  is constructed in \cite{ag} and  whose construction depends
on the construction of group relation algebras and the results in
this paper.    In \cite{ga},   an analysis of atomic, measurable
relation algebras is carried out, and it is proved that every
atomic, measurable relation algebra is essentially isomorphic to a
coset relation algebra. If the given algebra is actually finitely
measurable, then the assumption of it being atomic is unnecessary.
The results were   announced without proofs in \cite{ga02}. Except
for basic facts about groups, the article is intended to be
self-contained.  For more information about relation algebras, the
reader may consult \cite{ga17},  \cite{ga18},  \cite{hh02}, or
\cite{ma06}.


\section{Complex algebras of groups}\label{S:2}

In the 1940's, J. C. C. McKinsey observed that the complex algebra
of a group is a \ra. Specifically,  let  $\langle G\smbcomma
\scir\smbcomma\mo\smbcomma e\rangle$ be a group and $\SB G$ the
collection of all subsets, or \textit{complexes}, of $G$. The group
operations of multiplication (or  composition) and inverse can be
extended to operations on complexes in the obvious way:
\[ H\scir K=\{ h\scir k : h\in H \text{ and } k\in
K\}\] and
\[H\mo= \{ h\mo : h\in H\}.
\]
In order to simplify notation, we shall often identify elements with
their singletons, writing, for example, $g\scir H$ for $\{ g\}\scir
H$, so that
\[ g\scir H =\{ g\scir h : h\in H\}.
\]

The collection $\SB G$ of complexes contains the singleton set $\{
e\}$ and is closed under the Boolean operations of union and
complement, as well as under the  group operations of complex
multiplication and inverse. Thus, it is permissible to form the
algebra
\[
\gcm{G}= \langle
 \SB{G}\smbcomma \cup\smbcomma\sim
\smbcomma\scir\smbcomma\mo\smbcomma\{ e\}\rangle,
\] and it is easy to check that this is a relation algebra.  In
fact, it is representable via a slight modification of the Cayley
representation of the group. In more detail, for each element $g$ in
$G$, let $\rr g$ be the binary relation on $G$ defined by
\[\rr g=
\{\pair h {h\scir g}:h\in G)\}\per\]  The correspondence
$g\longmapsto\rr g$ is a slightly modified version of the Cayley
representation of $G$ as a group of permutations in which the
operation of relational composition is used instead of functional
composition. In particular,
\begin{alignat*}{3}
\rr g&=\idd G&&\qquad\text{if and only if}\qquad &g &= e,\\
\rr g\mo&=\rr k&&\qquad\text{if and only if}\qquad &g\mo &= k,\\ \rr
f\rp\rr g&=\rr k&&\qquad\text{if and only if}\qquad &f\scir g &=k.
\end{alignat*}
For each subset $X$ of  $G$, write
\[S_X=\bbigcup {g\in X}\rr g,\] and take $A$ to be
 the set of all  relations $S_X$ for $X\seq G$.  Using
the properties of the relations $\rr g$ displayed above, and also
the complete distributivity of the operations of relational
composition and converse over unions, it is a simple matter to check
that $A$ is a subuniverse of  the set relation algebra $\full E$
with $E$ the universal relation on the set $G$,   so that the
correspondence mapping each set $X$ to the relation $\cs SX$ is an
embedding of $\com G $ into $\full E$. We shall call this mapping
the \textit{Cayley representation} of $\gcm G$.

There is a natural extension of the Cayley representation of a group
$G$ to a representation of a quotient group $G/H$. If
$\langle\h\g:\g<\kappa\rangle$ is a coset system for a normal
subgroup $H$ of $G$, then define the representative of a coset
$\h\lph$ to be the binary relation
\[\rr \lph=\bbigcup
{\g<\ka}\h\g\times(\h\g\scir\h\lph).\] (To minimize the number of
parentheses that are used, we  adopt here and everywhere below the
standard convention that multiplications---in this case Cartesian
products---take precedence over additions---in this case, unions.)
Notice that, strictly speaking, $\rr\lph$ is not the Cayley
representation of $\h\lph$, which is the set of ordered pairs
\[\{\pair{\h\g}{\h\g\scir\h\lph}:\g<\kappa\}\per\]

The notion of a relation representing a coset can be taken one step
further. If $\varphi$ is an isomorphism from a quotient group $G/H$
to another quotient group~$F/K$, then~$F/K$ is identical with $G/H$
except for the ``shape" of its elements, and therefore it makes
sense to identify each coset $\h\g$ in $G/H$ with its image
$\varphi(\h\g)=K_{\g}$ in~$F/K$.  One can then take the
representative of a coset $\h\lph$ to be the relation
\[\rr \lph=\bbigcup
{\g<\ka}\h\g\times\varphi(\h\g\scir\h\lph) =\bbigcup
{\g<\ka}\h\g\times(K_{\g}\scir K_{\lph}).\]  Notice that each
relation $\rr\lph$ is a union of rectangles, that is to say, it is a
union of relations of the form $X\times Y$, and these rectangles are
mutually disjoint, because the cosets $\h\g$ are mutually disjoint,
as are the cosets $K_{\g}\scir K_{\lph}$, for distinct $\g<\ka$.

 To illustrate this idea with a concrete example, consider the two
 groups   $\mathbb Z_6$ and~$\mathbb Z_9$ (the integers modulo $6$ and the integers modulo $9$), and the canonical
 isomorphism $\varphi$ between the quotients
\[\mathbb Z_6/\{0,3\}\qquad\text{and}\qquad \mathbb Z_9/\{0,3,6\}\]  that maps the cosets
\begin{gather*}
H_0 =\{0,3\} \quad\text{to}\quad  K_0 =\{0,3, 6\},\quad H_1
=\{1,4\}\quad\text{to}\quad K_1 =\{1,4, 7\},\\ H_2
=\{2,5\}\quad\text{to}\quad K_2 =\{2,5, 8\}.
\end{gather*}
Using this correspondence,  define three relations as follows:
\begin{align*}  R_0&=[H_0\times (K_0\scir K_0)]\cup[H_1\times
(K_1\scir K_0)]\cup[H_2\times (K_2\scir K_0)]\\ &=[H_0\times
K_0]\cup[H_1\times K_1]\cup[H_2\times K_2]\\ &=\{(a,b): a\in \mathbb
Z_6\text{\ ,\ }b\in \mathbb Z_9\text{ and } b\equiv a\text{ mod }
3\}\comma
\\ R_1&=[H_0\times (K_0\scir K_1)]\cup[H_1\times
(K_1\scir K_1)]\cup[H_2\times (K_2\scir K_1)]\\ &=[H_0\times
K_1]\cup[H_1\times K_2]\cup[H_2\times K_0]\\ &=\{(a,b): a\in \mathbb
Z_6\text{\ ,\ }b\in \mathbb Z_9\text{ and } b\equiv a+1\text{ mod }
3\}\comma\\ R_2&=[H_0\times (K_0\scir K_2)]\cup[H_1\times (K_1\scir
K_2)]\cup[H_2\times (K_2\scir
K_2)]\\ &=[H_0\times K_2]\cup[H_1\times K_0]\cup[H_2\times K_1]\\
&=\{(a,b): a\in \mathbb Z_6\text{\ ,\ }b\in \mathbb Z_9\text{ and }
b\equiv a+2\text{ mod } 3\}\per
\end{align*}
\begin{figure}[tbh]
\includegraphics*[scale=.8]{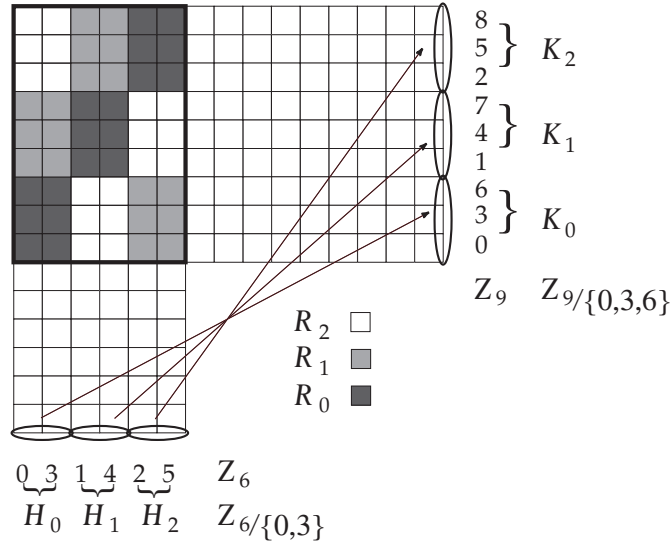}
\caption{The relations $R_0$, $R_1$, and $R_2$.}\label{F:fig0}
\end{figure}
(See Figure~\ref{F:fig0}.)  The relations $R_0$, $R_1$, and $R_2$
are
 \textit{representatives} of the cosets $H_0$, $H_1$,
and $H_2$ respectively, and together they give a kind of
representation of $\mathbb Z_3$ that has the flavor of the Cayley
representation of $\mathbb Z_3$. (Notice, however, that this is not
a real representation of $\mathbb Z_3$, since we cannot form the
composition of these relations.) This  is a key idea  in the
construction of measurable algebras of binary relations  from
\textit{systems}  of groups and quotient isomorphisms.

\section{Systems of groups and quotient isomorphisms}\label{S:3}

Fix a system
\[G=\langle
\G x:x\in I\,\rangle\] of  groups $\langle \G
x\smbcomma\scir\smbcomma\mo\smbcomma \e x\rangle$ that are pairwise
disjoint, and an associated system
\[\varphi=\langle\vph_{xy}:\pair x y\in \mc E\,\rangle\] of
quotient isomorphisms.  Specifically,   $\mc E$ is an equivalence
relation on the index set $I$, and for each pair $\pair x y$ in $\mc
E$, the function $\vphi {xy}$  is an isomorphism from a quotient
group of $\G x$ to a quotient group of $\G y$.  We shall call
\[\mc F=\pair G \varphi\]  a \textit{group pair}.  The set    $I$ is
the \textit{group index set},  and the equivalence relation $\mc E$
is the (\textit{quotient}) \textit{isomorphism index set}, of $\mc
F$. The normal subgroups of $\G x$ and $\G y$ from which the
quotient groups are constructed are uniquely determined by
 $\vphi\xy$, and will be denoted  by $\h\xy$ and $K_{xy}$
respectively, so that $\vphi\xy$ maps $\gxhi$ isomorphically onto
$\gyki$.

For a fixed enumeration $\langle \hs\wi \g:\g<\kai\wi\rangle$
(without repetitions) of the cosets of $\h\wi$ in $\G x$ (indexed by
some ordinal number $\kai\xy$), the isomorphism $\vphi\wi$ induces a
\textit{corresponding}, or \textit{associated},  coset system
 of $K_{\wi}$ in $\G y$,  determined by the rule
\[\ks\wi \g=\vphi\wi(\hs\wi \g)\]
for each $\gm<\kai\wi$.
 In what follows we  shall always assume that the given coset
systems for $\h\wi$ in $\G x$ and for $K_{\wi}$ in $\G y$ are
associated in this manner. Furthermore, there is no loss of
generality in assuming that the first elements in the enumeration of
the coset systems are always the normal subgroups themselves, so
that
\[\hs\wi 0 =\h\wi\qquad\text{and}\qquad\ks\wi 0 =K_{\wi}.\]
\begin{definition} \label{D:compro}  For each pair $\pair
x y$  in $\mc E$ and each $\alpha <\kai\wi$, define a  binary
relation $R_{{\wi},{\alpha}}$ by\[ R_{{\wi},{\lph}}= \tbigcup_{\gm <
\kp_{\wi}} H_{\wi,\gm}\times \vphi\wi[H_{\wi,\gm}\scir
H_{\wi,\lph}]= \tbigcup_{\gm < \kp_{\wi}} H_{\wi,\gm}\times
(K_{\wi,\gm}\scir K_{\wi,\lph})\per\] \qeddef\end{definition}

The index $\lph$ enumerating the relations $R_{\xy,\lph}$ coincides
with the index enumerating the coset system for the subgroup
$\h\xy$, and therefore is dependent upon the particular, often
arbitrarily chosen, enumeration of the cosets.   It would be much
better if the index enumerating the relations were independent of
the particular coset system that has been employed. This can be
accomplished by using the cosets themselves as indices, writing, for
instance, for each coset $L$ of $\h\xy$, that is to say, for each
element $L$ in $\G x/\h\xy$,
\[R_{\xy,{L}}=\tbigcup\{H\times\vp(H\scir L):H\in\G x/\h\xy\}
\]
instead of $R_{\xy,\lph}$.  In fact, it is really our intention that
the relations be indexed by the cosets and not by the  indices of
the cosets.  However, adopting this notation in practice eventually
becomes notationally a bit unwieldy. For that reason, we shall
continue to use the coset indices $\lph$, but we view these only as
convenient abbreviations for the cosets themselves. In places where
the distinction is important, we shall point it out.

Notice that the relation $R_{\xy, 0}$ encodes the isomorphism
$\vphi\xy$.

 In  proofs, we shall use repeatedly the fact that
operations such as forward and inverse images of sets under
functions, Cartesian multiplication of sets, intersection of sets,
complex group composition, relational composition, and relational
converse are all distributive over arbitrary unions, and we shall
usually simply refer to this fact by citing \textit{distributivity}.

\begin{lemma}[Partition Lemma] \label{L:i-vi}  The
relations $R_{\wi, \alpha} $\comma for $\alpha <\kai\xy$\comma are
non-empty and partition the set $\gsq x y$.
\end{lemma}

\begin{proof}
Obviously, the relations  are non-empty, because the cosets used to
construct them are non-empty.
 The sequence $\langle H_{\wi,\gm} : \gm < \kp_{\wi}\,\rangle$ is a
coset system for $\h\wi$ in  $G_\wx$, so these cosets are mutually
disjoint and have $\cs G\wx$ as their union. Similarly, the cosets
in the corresponding sequence $\langle K_{\wi,\gm} : \gm <
\kp_{\wi}\,\rangle$ are mutually disjoint and have $\cs G\wy$ as
their union. The sequence obtained by multiplying each $\ks\wi \g$
on the right by a fixed coset $\ks\wi \alpha$ lists the cosets of
$K_{\wi}$ in some permuted order.   These observations and the
distributivity of Cartesian multiplication yield
\begin{align*}
\tbigcup_\lph R_{{\wi},{\lph}} &= \tbigcup_\lph\tbigcup_\gm
H_{\wi,\gm}\times (K_{\wi,\gm}\scir K_{\wi,\lph})  =
\tbigcup_\gm\tbigcup_\lph H_{\wi,\gm}\times (K_{\wi,\gm}\scir
K_{\wi,\lph})\\ &= \tbigcup_\gm H_{\wi,\gm}\times
\bigl(\tbigcup_\lph K_{\wi,\gm}\scir K_{\wi,\lph}\bigr) =
\tbigcup_\gm H_{\wi,\gm}\times G_\wy\\  &= \bigl(\tbigcup_\gm
H_{\wi,\gm}\bigr)\times G_\wy = G_\wx\times G_\wy.
\end{align*}

The cosets $H_{\wi,\gm}$ and $H_{\wi,\dlt}$ are  disjoint whenever
$\g\neq\dlt$, and so are
  the cosets $K_{\wi,\lph}$ and $K_{\wi,\bt}$---and therefore also
the  cosets $K_{\wi,\dlt}\scir K_{\wi,\lph}$ and $K_{\wi,\gm}\scir
K_{\wi,\bt}$---whenever $\lph\neq\bt$. Consequently,
\begin{equation*}
[H_{\wi,\gm}\cap H_{\wi,\dlt}] \times [(K_{\wi,\gm}\scir
K_{\wi,\lph})\cap (K_{\wi,\dlt}\scir K_{\wi,\bt})]=\varnot\tag{1}
\end{equation*}
whenever $\g\neq\dlt$ or $\alpha \neq \beta$. For distinct
$\lph,\bt$, a simple computation leads to
\begin{align*}
R_{\wi,\bt} &= \big[ \tbigcup_\gm H_{\wi,\gm} \times
(K_{\wi,\gm}\scir K_{\wi,\lph})\big]\cap \big[ \tbigcup_\dlt
H_{\wi,\dlt} \times (K_{\wi,\dlt}\scir K_{\wi,\bt})\big] \\ &=
\tbigcup_{\gm,\dlt}
 [H_{\wi,\gm} \times (K_{\wi,\gm}\scir K_{\wi,\lph})]\cap
[H_{\wi,\dlt} \times (K_{\wi,\dlt}\scir K_{\wi,\bt})] \\ &=
\tbigcup_{\gm,\dlt}[H_{\wi,\gm}\cap H_{\wi,\dlt}] \times
[(K_{\wi,\gm}\scir K_{\wi,\lph})\cap (K_{\wi,\dlt}\scir
K_{\wi,\bt})] \\ &=\varnot,
\end{align*}
by the definition of $R_{\wi, \alpha}$ and $R_{\wi, \beta}$, the
distributivity of  intersection  and Cartesian multiplication, and
(1).
\end{proof}

Let $U$ be the union of the disjoint system of groups,  and $E$ the
equivalence relation on $U$ induced by the isomorphism index set
$\mc E$, \[U=\tbigcup\{\cs Gx:x\in I\}\qquad\text{and}\qquad
E=\tbigcup\{\cs Gx\times\cs Gy:\pair xy\in \mc E\}\per\] Write
\[\mc I=\{ (\pair x y,\lph) : \pair x y\in \mc E \text{ and }\lph
< \kp_{\wi}\}
\] for the \textit{relation index set} of the group pair $\mc F$, that
is to say, for the set of indices of the relations $R_{\wi,
\alpha}$. For each subset $\mc X$ of   $\mc I$, define
\[
S_{\mc X}= \tbigcup \{ R_{{\wi},{\lph}} : ((x,y),\lph)\in {\mc X}\},
\] and let $A$ be the collection of all of the relations $\cs
S{\mc X}$ so defined.

\begin{theorem}[Boolean Algebra Theorem]\label{T:disj} The set $A$
is the universe of a complete and atomic Boolean algebra of subsets
of $E$\per The distinct elements in $A$ are the relations $\cs S{\mc
X}$ for distinct subsets $\mc X$ of $\mc I$\comma and the atoms are
the relations $R_{{\wi},{\lph}}$ for $(\pair x y,\lph)$ in $\mc
I$\per The unit is the relation $E=S_{\mc I}$, and the operations of
union, intersection, and complement in $A$ are determined by
\begin{equation*}\tbigcup_\xi S_{{\mc X}_\xi} = S_{\mc Y},\qquad \tbigcap_\xi S_{{\mc X}_\xi}
= S_{\mc Y},\qquad S_\mc I\setcomp S_{\mc X} =   S_{\mc Y}
\end{equation*} where ${\mc Y}=\tbigcup_\xi {\mc X}_\xi$ in the
first case\comma ${\mc Y}=\tbigcap_\xi {\mc X}_\xi$ in the second
case\comma and ${\mc Y}=\mc I\setdiff {\mc X}$ in the third case
\opar for any system  $(\cs{\mc X}\xi:\xi<\lambda)$ of subsets\comma
and  any subset $\mc X$\comma of $\mc I$\cpar\per
\end{theorem}
\begin{proof}
  The system of rectangles $\langle\cs Gx\times\cs Gy :\pair xy\in\mc
E\rangle$ is easily seen to be a partition of $E$.  Combine this
with Lemma~\ref{L:i-vi} and  the definition of the relations $\cs
S{\mc X}$ to arrive at the desired result. \end{proof}

Although the set $A$ is always a complete Boolean set algebra of
binary relations, it is not in general closed under the operations
of relational composition and converse, nor does it necessarily
contain the identity relation $\id U$ on the set $U$.  Such closure
depends on the properties of the quotient isomorphisms.  We begin by
characterizing when $A$ contains $\id U$.

\begin{theorem}[Identity Theorem]\label{T:identthm1} For each element $x$ in
$I$\co the following conditions are equivalent\per
\begin{enumerate}
\item[(i)]
The identity relation $\idd {\G x}$ on $\G\wx$ is in $A$\per
\item[(ii)] $R_{\xx, 0}=\idd {\G x}$\per
\item[(iii)]$\vphi\xx$ is the identity
automorphism of $\G\wx/\{\ex \wx\}$\po
\end{enumerate}
Consequently\co $A$ contains the identity relation $\id U$ on the
base set $U$ if and only if \textnormal{(iii)}  holds for each $\wx$
in $I$\po
\end{theorem}

\begin{proof}   Suppose  (i) holds, with the intention of deriving  (iii).
From the assumption in (i), and the definition of the set $A$, it is
clear that $\idd {\G x}$ must be a (non-empty) union of some of the
relations  $R_{\wj,\alpha}$. Each relation $R_{\wj,\alpha}$ in such
a union is a subset of the rectangle $\G y\times\G z$, by Partition
Lemma~\ref{L:i-vi},  and it is simultaneously a subset of the square
$\G x\times \G x$, because $\idd {\G x}$ is  a subset of  $\G
x\times\G x$.  The rectangle and the square are disjoint  whenever
$x\neq y$ or $x\neq z$, so $x=y=z$, and therefore
\begin{equation*}\tag{1}\label{Eq:idt1.1}
\tbigcup_{\g}\hs\xx \g\times (\ks\xx \g\scir\ks\xx \alpha)=
R_{\xx,\alpha} \seq \idd {\G x}= \tbigcup \{\pair gg:g\in \G
x\}\comma
\end{equation*} by the definitions of $R_{\xx,\alpha}$ and $\idd {\G
x}$. This inclusion implies that  the cosets $\hs \xx\g$ and $\ks\xx
\g\scir\ks\xx \alpha$ on the left side of \eqref{Eq:idt1.1} contain
exactly one element each, and this element  is the same for both
cosets, for if this were not the case, then the  Cartesian product
of the two cosets would contain a pair of the form $\pair gh$ with
$g\neq h$, in contradiction to \eqref{Eq:idt1.1}.  Thus, for each
$\gm<\kai\xx$, there is an element $g$ in $\G x$ such that
\begin{equation*}\tag{2}\label{Eq:idt1.2}
 \hs\xx \gm = \{g\}\qquad\text{and}\qquad \ks\xx
\g\scir\ks\xx \alpha = \{g\}\per
\end{equation*}

Take $\gm=0$ in \eqref{Eq:idt1.2}, and   apply  the convention that
$\hs\xx 0$ and $\ks\xx 0$ coincide with the subgroups $\h\xx$ and
$K_{\xx}$ respectively;  these subgroups are the identity cosets of
the quotient groups $\G x/\h\xy$ and $\G y/K_{\xy}$, so
\begin{equation*}\tag{3}\label{Eq:idt1.4}
\h\xx= \hs\xx 0 = \{g\}\qquad\text{and}\qquad \ks\xx\alpha=\ks\xx
0\scir\ks\xx \alpha = \{g\}\per
\end{equation*}
By assumption $\h\xx$ is a normal subgroup of $\G x$.  The only
normal subgroup that has exactly one element is the trivial subgroup
$\{e_x\}$, so the element $g$  in \eqref{Eq:idt1.4} must coincide
with $e_x$.  Use the right side of \eqref{Eq:idt1.4} with $g=e_x$ to
see that $\alpha$ must be $0$.

 Invoke \eqref{Eq:idt1.2} one more time to obtain, for each $\g<\kai\xx$,   an element $g$ in $\G x$ such
 that   \begin{equation*}\tag{4}\label{Eq:idt1.3} \hs\xx \g
=\{g\}=\ks\xx \g\scir\ks\xx \alpha=\ks\xx \g\scir\ks\xx 0=\ks\xx
\g\per
\end{equation*}
The isomorphism $\vphi\xx$  is assumed to map  $\hs\xx \g$ to
$\ks\xx \g$ for each $\g$, so \eqref{Eq:idt1.3} shows that $\vphi
\xx$ maps each singleton $\{g\}$ to itself. It follows that
$\vphi\xx$ is the identity isomorphism on $\G x/\{\e x\}$. Thus,
(iii) holds.

If (iii) holds, then \[R_{\xx, 0}=\tbigcup
\{\{g\}\times(\{g\}\scir\{ e_x\}):g\in\G x\}=\{\pair gg:g\in \G
x\}=\id{\cs Gx}\comma\] by the definition of $R_{\xx,0}$, so (ii)
holds. On the other hand, if (ii) holds, then (i) obviously holds,
by the definition of $A$.

 To derive the final assertion of the theorem, assume first that
(iii) holds. The identity relation $\idd {\G x}$ is then in $A$, by
(i). The union, over all $x$, of these identity relations is the
identity relation $\id U$. Since $A$ is closed under arbitrary
unions, it follows that $\id U$ is in $A$.

 Now assume that $\id U$
is in $A$. The squares $\cs Gx\times \cs Gx$ are all in $A$, by
Lemma~\ref{L:i-vi} and the definition of $A$, so the intersection of
each of these squares with $\id U$ is in $A$, by the closure of $A$
under intersection. This intersection
 is just $\idd {\G x}$, so (i) holds, and therefore also (iii),
for each $x$.
\end{proof}

In order to prove the next two theorems, it is convenient to
formulate two lemmas that will be used in both proofs.

\begin{lemma} \label{L:rect2} Suppose that each of  \[\langle
M_\lph:\lph<\kappa\rangle\comma\qquad \langle
N_\lph:\lph<\kappa\rangle\comma\qquad \langle
P_\bt:\bt<\lambda\rangle\comma\qquad\langle
Q_\bt:\bt<\lambda\rangle\] are sequences of non-empty\comma pairwise
disjoint sets\po  If
\begin{enumerate}
\item[(i)]    $\bbigcup {\lph<\ka}M_\lph\times N_\lph\seq
\bbigcup {\bt<\la}P_\bt\times Q_\bt$\rrcomma
\end{enumerate} then there is a
uniquely determined mapping $\vth$ from $\ka$ into $\la$ such that
\begin{enumerate}\item[(ii)] $M_\lph\seq P_{\vth(\lph)}\qquad
\text{and}\qquad N_\lph\seq Q_{\vth(\lph)}$\end{enumerate}
 for each $\lph<\ka$\po If equality holds in \textnormal{(i)}\comma
then equality holds in \textnormal{(ii)}\comma and $\vth$ is a
bijection\per
\end{lemma}
\begin{proof} Consider, first, arbitrary non-empty sets $M$ and $N$.  Assume \begin{equation*}\tag{1}\label{Eq:rect1.1}
M\times N=\tbigcup_{\bt<\la} P_\bt\times Q_\bt\comma \end{equation*}
with the intention of proving that   $\la=1$  (recall that $\la$ is
an ordinal), and \begin{equation*}\tag{2}\label{Eq:rect1.2}M=
P_0\qquad  \text{and}\qquad N= Q_0\per\end{equation*} It is obvious
from \eqref{Eq:rect1.1} that $P_\bt\times Q_\bt\seq M\times N$, and
therefore $P_\bt\seq M$ and $Q_\bt\seq N$,  for each $\bt<\la$.
Consequently,
\begin{equation*}\tag{3}\label{Eq:rect1.3}
  \bbigcup {\bt<\la} P_\bt\seq M\qquad\text{and}\qquad
\bbigcup {\bt<\la} Q_\bt\seq N\per
\end{equation*}  Use the distributivity of Cartesian multiplication, \eqref{Eq:rect1.3}, and \eqref{Eq:rect1.1} to obtain
\begin{equation*} \bbigcup {\lph,\bt<\la}  P_\lph\times Q_\bt
=\bigl(\bbigcup {\bt<\la} P_\bt\bigr)\times \bigl(\bbigcup {\bt<\la}
Q_\bt\bigr)  \seq M\times N = \bbigcup {\g <\la} P_\g\times Q_\g\per
\tag{4}\label{E:E1}\end{equation*} The inclusion of the first union
in the last one in \ref{E:E1} implies that every pair $\pair gh$ in
a rectangle $P_\bt\times Q_\beta$ must belong to some rectangle
$P_\g\times Q_{\g}$.  This cannot happen if $\alpha \neq \beta$,
because in such a case either $\alpha\neq\g$ or $\beta\neq\g$, and
therefore either $P_{\alpha}$ must be disjoint from $P_\g$, or else
$Q_\beta$ must be disjoint from $Q_\g$.  It follows that there is
exactly one $\beta$ that is less than $\la$. Since $\la $ is assumed
to be an ordinal, this forces $\la=1$ and $\beta=0$. Thus,
\eqref{Eq:rect1.1} assumes the form
\[M\times N=P_0\times Q_0\comma\] and clearly, \eqref{Eq:rect1.2} holds in this case.

Next, suppose that the equality in \eqref{Eq:rect1.1} is replaced
with set-theoretic inclusion, so that
\begin{equation*}\tag{5}\label{Eq:rect1.5}
M\times N\seq\tbigcup_{\bt<\la} P_\bt\times Q_\bt\per\end{equation*}
There is then  a unique index $\bt<\la$  such that
\begin{equation*}\tag{6}\label{Eq:rect1.6}M\seq P_\bt\qquad
\text{and}\qquad N\seq Q_\bt\per\end{equation*} For the proof, form
the intersection of both sides of \eqref{Eq:rect1.5} with $M\times
N$, and use \eqref{Eq:rect1.5},  the distributivity of intersection,
and simple set theory to obtain
\begin{multline*}\tag{7}\label{Eq:rect1.7}
 M\times N=(M\times N)\cap(M\times N)=(M\times N)\cap[\bbigcup {\bt<\la}  (P_\bt \times
Q_\bt)]\\
\bbigcup {\bt<\la} [(M\times N)\cap (P_\bt \times Q_\bt)]=\bbigcup
{\bt<\la} (M\cap P_\bt)\times (N\cap Q_\bt) \per\end{multline*}
 Drop all terms in the union on the right side of \eqref{Eq:rect1.7} that are empty.
 The equality of the first and last expressions in \eqref{Eq:rect1.7}
 shows that  \eqref{Eq:rect1.1} holds with $P_\beta$ and $Q_\beta$ replaced by $M\cap P_\beta$ and $N\cap Q_\beta$
 respectively.  Use the implication from \eqref{Eq:rect1.1} to \eqref{Eq:rect1.2} to conclude that
 there can only be one index $\beta$ on the right side of  \eqref{Eq:rect1.7}  for which the intersection
 is not empty, and for that $\beta$ we have
 \[M=M\cap P_\beta\qquad\text{and}\qquad N=N\cap Q_\bt\comma
 \] so that  \eqref{Eq:rect1.6} holds.

 Turn now to the proof of the implication from (i) to (ii). Fix an arbitrary index
$\lph<\ka$. From (i),  it follows immediately that
\[M_\lph\times N_\lph\seq
\bbigcup {\bt<\la}P_\bt\times Q_\bt\per \] Apply the implication
from \eqref{Eq:rect1.5} to \eqref{Eq:rect1.6} to obtain a unique
$\bt<\la$ such that
\begin{equation*}\tag{8}\label{Eq:rect.1.new}
  M_\lph\seq  P_{\bt}\qquad\text{and}\qquad  N_\lph\seq  Q_{\bt}\per
\end{equation*}
  The desired function
is the mapping $\vth$ that sends $\alpha$ to the corresponding
$\bt$, so that \eqref{Eq:rect.1.new} holds for each $\alpha<\kappa$.

Assume finally that equality holds in (i). There are then  uniquely
determined mappings $\vth$ from $\ka $ to $\la$ and $\psi$ from
$\la$ to $\ka$ such that
\begin{align*}
  M_\lph\seq P_{\vth(\lph)}\qquad&\text{and}\qquad N_\lph\seq
Q_{\vth(\lph)}\tag{9}\label{Eq:rect1.8} \\ \intertext{ for each
$\lph<\ka$, and}
 P_\bt\seq M_{\psi(\bt)}\qquad&\text{and}\qquad Q_\bt\seq
N_{\psi(\bt)}\tag{10}\label{Eq:rect1.9}\\ \intertext{for each
$\bt<\la$. Combine \eqref{Eq:rect1.8} and \eqref{Eq:rect1.9} to
arrive at} M_\lph\seq P_{\vth(\lph)}\seq M_{\psi(\vth(\lph))}
\qquad&\text{and}\qquad P_\bt\seq M_{\psi(\bt)}\seq
P_{\vth(\psi(\bt))}\tag{11}\label{Eq:rect1.10}
\end{align*}
       for each $\lph<\ka$ and $\bt<\la$.
The sets $M_\lph$ are pairwise disjoint, as are the sets $P_\bt$, so
the  inclusions in \eqref{Eq:rect1.10} force
\[{\psi(\vth(\lph))}=\lph\qquad\text{and}\qquad
{\vth(\psi(\bt))}=\bt\] for each $\lph<\ka$ and $\bt<\la$. This
implies that the mappings $\vth$ and $\psi$ are bijections and
inverses of each other.
\end{proof}

\begin{lemma} \label{L:PQnormsg}
Suppose $P$ and $Q$ are normal subgroups of groups $G$ and $\bar
G$\co with coset systems \[\langle P_\g : \g <
\ka\rangle\qquad\text{and}\qquad \langle Q_\g ; \g < \ka\rangle\]
respectively\po If the mapping $P_\g \longmapsto Q_\g$ is an
isomorphism from $\gp$ onto $\bar G/Q$\comma then for all
$\alpha,\bt<\ka$\comma we have
\begin{itemize}
\item[\opar i\cpar] $\tsbigc_\gm P_\gm \times (Q_\gm\scir
Q_\lph)= \tsbigc_\gm (P_\gm\scir P_\lph\mo)\times Q_\gm$\rrcomma
\item[\opar ii\cpar] $\tsbigc_\gm P_\gm \times (Q_\gm\scir
Q_\lph\scir Q_\bt)= \tsbigc_\gm (P_\gm\scir P_\lph\mo)\times
(Q_\gm\scir Q_\bt)$\po
\end{itemize}
\end{lemma}
\begin{proof}
Fix an index $\alpha<\ka$, and observe that  $ \langle P_\gm\scir
P_\lph\mo : \gm < \ka\rangle$ is also an enumeration of the cosets
of $P$. Consequently, for each $\gm < \ka$ there exists a unique
$\bar\gm  < \ka$ such that
\begin{equation*} \tag{1}\label{Eq:P.1}
P_\gm= P_{\bar\g }\scir P_\lph\mo\per\end{equation*} The mapping
$P_\g\longmapsto Q_\g$ is assumed to be an isomorphism, so (1)
implies that
\begin{equation*}\tag{2}\label{Eq:Q.1}
Q_\gm= Q_{\bar\g }\scir Q_\lph\mo.
\end{equation*}
Use \eqref{Eq:P.1}, \eqref{Eq:Q.1}, and the inverse properties of
groups to get
\begin{equation*}\tag{3}\label{Eq:P.2}
\tsbigc_\gm P_\gm \times (Q_\gm\scir Q_\lph) = \tsbigc_\gm
(P_{\bar\g } \scir P_\lph\mo)\times (Q_{\bar\g }\scir Q_\lph\mo
\scir Q_\lph) = \tsbigc_\gm (P_{\bar\g } \scir P_\lph\mo)\times
Q_{\bar\g }\per
\end{equation*}
As $\g$ varies over $\ka$, so does $\bar\g$, and vice versa, so the
occurrence of $\bar\g$ in the union on the right side of
\eqref{Eq:P.2} may be replaced by $\g$ to arrive at (i).

Exactly the same reasoning also gives
\begin{align}
\tsbigc_\gm P_\gm \times (Q_\gm\scir Q_\lph\scir Q_\bt) &=
\tsbigc_\gm (P_{\bar\g } \scir P_\lph\mo)\times (Q_{\bar\g }\scir
Q_\lph\mo \scir Q_\lph\scir Q_\bt)\notag\\
&= \tsbigc_\gm (P_{\bar\g } \scir P_\lph\mo)\times (Q_{\bar\g }\scir
Q_\bt)\notag\\
&= \tsbigc_\gm (P_\gm \scir P_\lph\mo)\times (Q_\gm\scir
Q_\bt)\comma\notag
\end{align}
which proves (ii).
\end{proof}

The next task is to establish  necessary and sufficient conditions
for the set $A$ to be closed under converse, and in particular, for
$A$ to contain the converse of every atomic relation. As we shall
see in the next theorem, $A$ will contain the converse of every
atomic relation if and only if the isomorphism $\vphi\yx$ is the
inverse of the isomorphism $\vphi\xy$ for every pair $\pair xy$ in
$\mc E$.  Since $\vphi\xy$ maps $\cs Gx/\h\xy$ to $\cs Gy/K_{\xy}$,
and $\vphi\yx$ maps $\cs Gy/\h\yx$ to $\cs Gx/K_{\yx}$, if these two
isomorphisms are  inverses of one another, then we must have
\begin{align*}
\cs Gx/\h\xy=\cs Gx/K_{\yx}\qquad&\text{and}\qquad \cs
Gy/K_{\xy}=\cs Gy/\h\yx\comma\\ \intertext{so that}
K_{\yx}=\h\xy\qquad&\text{and}\qquad K_{\xy}=\h\yx\per
\end{align*}
As mentioned earlier,  the enumeration  of the cosets of the
subgroup $\h\yx$ can be  chosen freely.  Under the given assumption,
we can and shall always adopt the following convention regarding the
choice of this enumeration.
\begin{con}\label{Co:convention} If $\vphi\xy$  and $\vphi\yx$  are inverses of
one another, then the coset enumeration $\langle \hs\yx
\g:\g<\kai\yx\rangle$ is chosen so that $\kai\yx=\kai\xy$ and
 $\hs\yx\g=\ks\xy\g$ for all $\g<\kai\xy$.  It then follows  that
 \[\ks\yx \g =\vphi\yx(\hs \yx\g)=\vphi\xy\mo(\ks \xy\g)=
\ \hs\xy \g\] for all $\g<\kai\xy$.\qed\end{con}

The next theorem  characterizes when $A$ is closed under converse.

\begin{theorem}[Converse Theorem]\label{T:convthm1}  For each
pair $\pair x y$  in $\mc E$\co the following conditions are
equivalent\per
\begin{enumerate}
\item[(i)] There are an $\alpha<\kai \xy$ and a $\bt<\kai\yx$ such that $R_{\xy,\lph}\mo=R_{\yx,\bt}$\per
\item[(ii)] For every $\alpha<\kai \xy$ there is a $\bt<\kai\yx$ such that $R_{\xy,\lph}\mo=R_{\yx,\bt}$\per
\item[(iii)]$\vphi\xy\mo=\vphi\yx$.
\end{enumerate}
Moreover\comma if one of these conditions holds\comma then we may
assume that $\kai \yx=\kai\xy$\comma and the index $\bt$ in
\textnormal{(i)} and \textnormal{(ii)} is uniquely determined by $
\hs\xy\lph\mo=\hs\xy\bt$\per The set $A$ is closed under converse if
and only
 if \textnormal{(iii)} holds for all    $\pair x y$ in $\mc E$\po
\end{theorem}

\begin{proof}Observe, first of all, that without using any of the hypotheses in (i)--(iii),
only the definition of the relation $R_{\xy,\lph}$,
Lemma~\ref{L:PQnormsg}(i), the distributivity of relational
converse, and the definition of relational converse, we have
\begin{multline*}\tag{1}\label{Eq:ct1.03}
R_{{\xy},{\lph}}\mo  = \bigl[\tbigcup_\gm H_{\xy,\gm}\times
(K_{\xy,\gm}\scir K_{\xy,\lph})\bigr]\mo
 = \bigl[\tbigcup_\gm (H_{\xy,\gm}\scir\hs\xy\alpha\mo)\times
K_{\xy,\gm}\bigr]\mo\\ = \tbigcup_\gm
\bigl[(H_{\xy,\gm}\scir\hs\xy\alpha)\times K_{\xy,\gm} \bigr]\mo =
\tbigcup_\gm K_{\xy,\gm}\times (H_{\xy,\gm}\scir\hs\xy\alpha\mo)\per
\end{multline*}

Assume now that (iii) holds, with the intention of deriving (ii).
Choose $\bt<\kai\xy$ so that
 \begin{equation*}\tag{2}\label{Eq:ct1.08} \hs\xy \beta = \hs\xy \lph\mo\per
\end{equation*}  In view of  assumption (iii),   Convention~\ref{Co:convention} may be applied to write $\kai\yx=\kai\xy$, and
\begin{equation*}\tag{3}\label{Eq:ct1.09} \hs\yx \g = \ks\xy \g\comma \qquad \ks\yx \g =
\hs\xy \g
\end{equation*} for each $\g<\kai\xy$. Use the definition of the relation $R_{\yx,\beta}$,
together with \eqref{Eq:ct1.09}, \eqref{Eq:ct1.08}, and
\eqref{Eq:ct1.03}, to conclude
that\begin{multline*}R_{\yx,\beta}=\tbigcup_\gm H_{\xy,\gm}\times
(K_{\yx,\gm}\scir K_{\yx,\bt})
= \tbigcup_\gm K_{\xy,\gm}\times (H_{\xy,\gm}\scir H_{\xy,\bt})\\
= \tbigcup_\gm K_{\xy,\gm}\times (H_{\xy,\gm}\scir
H_{\xy,\lph}\mo)=R_{\xy,\lph}\mo\per
 \end{multline*}
 Thus, (ii) holds.

The implication from (ii) to (i) is obvious. Consider now the
implication from (i) to (iii). Fix $\alpha<\kai\xy$, and suppose
that
\begin{equation*}R_{\xy,\bt}=
R_{\xy,\alpha}\mo\per\tag{4}\label{Eq:new6}
\end{equation*} Use \eqref{Eq:new6}, the definition of $R_{\yx,\bt}$, and \eqref{Eq:ct1.03} (with $\gm$
replaced by another variable, say $\eta$) to obtain
\begin{equation*}
\tbigcup_{\g <\kai\yz}\hs\yx\g\times (\ks\yx \g\scir\ks\yx\bt)=
\tbigcup_{\eta<\kai\xy} K_{\xy,{\eta}}\times( H_{\xy,{\eta}}\scir
H_{\xy,\lph}\mo)\per\tag{5}\label{Eq:new8}
\end{equation*} Apply Lemma~\ref{L:rect2} to \eqref{Eq:new8} to see that there must be a bijection
 $\vth$ from $\kai\xy$ to $\kai\xy$  such that
\begin{equation*}
 \hs\yx \g=\ks\xy{\vth(\g)}\qquad\text{and}\qquad \ks\yx \g\scir\ks\yx\bt =
 \hs\xy
{\vth(\g)}\scir\hs\xy\bt \tag{6}\label{Eq:new8.9}
\end{equation*} for all $\g<\kai\yx$.

Take $\g=0$ in \eqref{Eq:new8.9}. It follows from the first equation
that $\hs\yx 0=\ks\xy{\vth(0)}$. Since $\hs\yx 0$ is a subgroup of
$\G y$, the same must be true of $\ks\xy{\vth(0)}$. The only
subgroup in the coset enumeration of $K_{\xy}$ is $\ks\xy 0$, so
$\vth(0)=0$, and therefore
\[\h\yx=\hs\yx 0=\ks\xy 0=K_{\xy}.\] Apply this observation to the second equation, and use
the fact that $\vth(0)=0$, to arrive at
\begin{multline*}\tag{7}\label{Eq:ct1.13}
\ks\yx\bt=K_{\yx} \scir\ks\yx\bt=\ks\yx 0\scir\ks\yx\bt\\=\hs\xy
{\vth(0)}\scir\hs\xy\bt=\hs\x 0\scir\hs\xy\bt=\h\xy
\scir\hs\xy\bt=\hs\xy\bt.
\end{multline*}
(Recall that $K_{\yx }$ and $\h\xy $ are the identity cosets of
their respective coset systems.)

Multiply the left and right sides of the second equation in
\eqref{Eq:new8.9}, on the right, by  $\ks\yx\bt\mo$, use the inverse
law for groups, and use the equality of the first and last cosets in
\eqref{Eq:ct1.13}, to arrive at
\begin{multline*}\tag{8}\label{Eq:new11}
\ks\yx\g=\ks\yx \g\scir\ks\yx\bt\scir \ks\yx\bt\mo = \hs\xy
{\vth(\g)}\scir\hs\xy\bt \scir \ks\yx\bt\mo\\
= \hs\xy {\vth(\g)}\scir\hs\xy\bt\scir  \hs\xy\bt\mo =\hs\xy
{\vth(\g)}
\end{multline*}
 for every $\g<\kai\yx$. Consequently,
\begin{equation*}
\vphi\yx(\ks\xy{\vth(\gm)})=\vphi\yx(\hs\yx\gm) = \ks\yx\gm
=\hs\xy{\vth(\gm)},
\end{equation*}
 by   \eqref{Eq:new8.9}, the definition of $\ks\yx\g$,  and \eqref{Eq:new11}. As $\gm$
runs through the indices less than $\kai\yx$, the image $\vth(\gm)$
runs through the indices less than $\kai\xy$, so the preceding
string of equalities shows that
\begin{equation*}
\vphi\yx(\ks\xy\dlt)=\hs\xy\dlt\tag{9}\label{Eq:new8.10}\end{equation*}
for every $\dlt<\kai\xy$.  Since $\vphi\xy$ maps each coset
$\hs\xy\dlt$ to $\ks\xy\dlt$, it follows from \eqref{Eq:new8.10}
that $\vphi\yx$ is the inverse of $\vphi\xy$.  This completes the
proof that conditions (i)--(iii) are equivalent.

If one of the three conditions holds, then all three conditions hold
by the equivalence just established.  Consequently, using the proof
of the implication from (iii) to (ii), we may assume that
$\kai\yx=\kai\xy$, and choose $\bt<\kai\xy$ so that
\eqref{Eq:ct1.08} holds. This proves the second assertion of the
theorem.

Turn  to the proof of the final assertion of the theorem. Assume
first that (iii) holds for all $\pair xy$ in $\mc E$. The
 atoms in $A$ are just the relations of the form $R_{\xy,\alpha}$, so
from the equivalence of (ii) with (iii), it follows that the
converse of every atom in $A$ is again an atom in $A$. The elements
of $A$ are just the unions of these various atoms, by
Theorem~\ref{T:disj}, and the converse of a union of atoms is again
a union of  atoms, by the preceding observation and the
distributivity of converse.  Thus, the converse of every element in
$A$ belongs to $A$, so $A$ is closed under converse.

Assume now that $A$ is closed under converse, and fix an arbitrary
pair $\pair x y$  in $\mc E$. The relation $R_{\xy,0}$ is a subset
of $\cs Gx\times\cs Gy$ and belongs to  $A$, by  Lemma~\ref{L:i-vi}
and the definition of $A$. It follows that the converse relation
$R_{\xy,0}\mo$ is a subset of $\cs Gy\times \cs Gx$, and it belongs
to $A$ by assumption. Consequently, there must be a non-empty set
$\Ga\seq\kai\yx$ such that
\begin{equation*}R_{\xy,0}\mo=\tbigcup_{\bt\in\Ga}R_{\yx,\bt}\comma\tag{10}\label{Eq:ct1.15}\end{equation*}
by Boolean Algebra Theorem~\ref{T:disj}. The pair $\pair {\e x}{\e
y}$ belongs to the relation $R_{\xy,0}$, by the definition of
$R_{\xy,0}$ (in fact, the pair is in $\hs\xy 0\times \ks\xy 0$,
which is one of the rectangles that make up $R_{\xy, 0}$), so the
converse pair $\pair {\e y}{\e x}$ belongs to $R_{\xy, 0}\mo$. For
similar reasons, the relation $R_{\yx, 0}$ contains the pair $\pair
{\e y}{\e x}$, and it is the only relation of the form $R_{\yx,
\bt}$ that contains this pair, because the atomic relations in $A$
are pairwise disjoint. It follows from this observation and
\eqref{Eq:ct1.15} that $0$ must be one of the indices in $\Ga$. In
other words,
\begin{equation*}\tag{11}\label{Eq:ct1.16}
  R_{\yx, 0}\seq R_{\xy, 0}\mo\per
\end{equation*} Reverse the roles of $x$ and $y$ in this argument
to obtain
\begin{equation*}\tag{12}\label{Eq:ct1.17}
  R_{\xy, 0}\seq R_{\yx, 0}\mo\per
\end{equation*} Combine \eqref{Eq:ct1.16} with \eqref{Eq:ct1.17}, and
use the monotony and first involution laws for converse, to arrive
at
\[R_{\xy, 0}\mo\seq(R_{\yx, 0}\mo)\mo=R_{\yx,0}\seq R_{\xy, 0}\mo\per\]  The first
and last terms are equal, so equality must hold everywhere.  In
particular, $R_{\xy, 0}\mo=R_{\yx, 0}$. This shows that condition
(i) is satisfied for the pair $\pair xy$ in the case $\alpha=0$.
Invoke the equivalence of (i) with (iii) to conclude that (iii)
holds for all pairs $\pair xy$.
\end{proof}

It is natural to ask  whether, in analogy with Identity
Theorem~\ref{T:identthm1}, one can add another condition to those
already listed in Converse Theorem~\ref{T:convthm1}, for example,
the condition that $R_{\xy, \alpha}\mo$ be  in $A$ for some
$\alpha<\kai\xy$. It turns out, however, that in the absence of
additional hypotheses, this condition is not equivalent to the
conditions listed in  the lemma. We   return to this question at the
end of the next section.

 Notice that condition (ii) in the preceding theorem, combined with the second assertion of the theorem, provides a
concrete method of computing the converse of a relation $R_{\xy,
\alpha}$ in terms of the structure of the quotient group $\G
x/\h\xy$: just compute the index $\bt$ such that $
\hs\xy\lph\mo=\hs\xy\bt$, for then we have
$R_{\xy,\lph}\mo=R_{\xy,\bt}$. This method, in turn, provides a
concrete way of computing the converse of any relation in $A$.

 The final and most difficult task is to characterize when  the
 set $A$ is closed under relational composition, and in
 particular, when it contains the composition of two atomic
 relations. There is one case in which the relative product of two
 atomic relations is empty, and therefore automatically in $A$.

 \begin{lemma}\label{L:emptycomp}  If $\pair \wx\wy$ and $ \pair w z$ are in $\mc E$\co and if $y\ne
w$\co then
\[R_{\xy,{\alpha}} \rp R_{\wwz, \bt}=\varnot
\]
for all $\lph<\kai\xy$ and $\bt <\kai\wwz$\po
\end{lemma}
\begin{proof}
Indeed,
\[
R_{{\xy},{\lph}}\seq G_{x}\times G_{y} \quad\text{and}\quad
R_{{\wwz},{\bt}}\seq G_w \times G_z,
\]
by Lemma~\ref{L:i-vi}.  Therefore,
\[
R_{{\xy},{\lph}} \rp R_{{\wwz},{\bt}}\seq (G_{x}\times G_{y}) \rp
(G_w \times G_z)\comma
\] by monotony.
If $y\neq w$, then the sets $G_{y}$ and $G_w $ are disjoint, and
therefore the relational composition of $G_{x}\times G_{y}$ and $G_w
\times G_z$ is empty.
\end{proof}

 To clarify the underlying ideas of the remaining case when $y=w$,
we again use cosets as indices of the atomic relations for a few
moments. It is natural to conjecture that (under suitable
hypotheses) the
 composition of the
 relations corresponding to  cosets   $H$ and $\bar H$ of $\h\xy$ and $\h\yz$ respectively
is precisely the relation corresponding to the   group composition
of the two cosets,
\[R_{\xy,{H}}\rp R_{\yz,{\bar H}}=R_{\xz,{H \scir\bar H}}.
\]

This form of the conjecture is  incorrect. The first difficulty is
that the cosets $H$ and $\bar H$ live in disjoint groups, and
therefore cannot be composed. To write the conjecture in a
meaningful way, one must first translate the coset $H$ to its copy,
the coset $K=\vphi\xy(H)$ of $K_{\xy}$ in $\G\wy$, where  $\bar H$
``lives",  and then compose this translation with $\bar H$ to arrive
at a coset \[M=K\scir\bar H\] of $K_{\xy}\scir\h\yz$.

 The second difficulty is that compositions  of  subrelations  of \[\cs
Gx\times \cs Gy\qquad\text{and}\qquad \cs Gy\times \cs Gz \] should
be a subrelation of $\cs Gx\times \cs Gz$, and therefore should have
$\xz$ as part of the index.  The relations  indexed with $\xz$ are
constructed with the help of cosets of $\h\xz$, so it is necessary
to translate the composite coset $K\scir\bar H$ from $\cs Gy$ back
to $\cs Gx$  using the mapping $\vphi\xy\mo$, so that it can be
written as a union of cosets of $\h\xz$. A more reasonable form of
the original conjecture might look like
\[R_{\xy,{H}}\rp R_{\yz,{\bar H}}=R_{\xz,{\vphi\xy\mo[\M]}}=R_{\xz,{\vphi\xy\mo[K\scir\bar H]}}\per
\]

 The third difficulty is that the relation on the right side of this last
equation has not been defined.   At this point, we can only speak in
a meaningful way about relations $R_{\xz,{\hat H}}$ for single
cosets ${\hat H}$ of $\h\xz$. It therefore is necessary to rewrite
the preceding conjecture in the form
\[R_{\xy,{H}}\rp R_{\yz,{\bar H}}=\tbigcup\{R_{\xz,{\hat H}} :{\hat H}\seq{\vphi\xy\mo[K
\scir\bar H]}\}.
\]

In order for the conjecture be true, the subgroup
$\vphi\xy\mo[K_{\xy}\scir\h\yz]$ must include the subgroup $\h\xz$,
so that  the coset $\vphi\xy\mo[K \scir\bar H]$
 can really be written as a union of cosets $\hat H$ of $\h\xz$.  Moreover, it is
natural to suspect that some sort of composition of the mappings
$\vphi\xy$ and $\vphi\yz$ should equal the mapping $\vphi\xz$,
\begin{alignat*}{2}&\ \vphi\xy&&\ \vphi\yz\\
\G x/\h\xy&\longmapsto\G y/K_{\xy}\comma\qquad&\qquad \G y/\h\yz
&\longmapsto\G z/K_{\yz},
\end{alignat*}
\begin{align*} &\vphi\xz\\
\G x/\h\xz&\longmapsto\G z/K_{\xz}.
\end{align*}  However, the subgroup
$K_{\xy}$ may not coincide with the subgroup $\h\yz$ at all, so it
is not meaningful to speak about the composition of $\ho x y$ with
$\ho y z$.  In order to be able to compose quotient isomorphisms,
one has first to form  a common quotient group using the complex
product of the subgroups, and then compose the induced isomorphisms
$\vphih \xy$ and $\vphih \yz$,
\begin{alignat*}{3}&\ \vphih\xy& &\ & &\ \vphih\yz\\
\G x/(\h\xy\scir\h\xz)&\longmapsto\ & \G y/&(K_{\xy}\scir\h\yz)&
&\longmapsto\G z/(K_{\xz}\scir K_{\yz}).\\ &\longmapsto& &\ \
\vphih\xz& &\longmapsto
\end{alignat*}
What really should be true is that the composition of the induced
mappings $\vphih\xy$ and $\vphih\yz$ should equal the induced
mapping $\vphih\xz$.
 These conditions do indeed prove to be necessary and sufficient
for the conjecture to hold. We  formulated them in the conventional
notation, using the subscripts of the cosets in place of the cosets.

\begin{theorem}[Composition Theorem]\label{T:compthm}  For
all pairs $\pair x y$ and $\pair y z$  in $\mc E$\co the following
conditions are equivalent\per
\begin{enumerate}
\item[(i)] The relation
$R_{\xy, 0} \rp R_{\yz, 0}$ is in $A$\per
\item[(ii)] For each
$\lph<\kai\xy$ and each $\bt<\kai\yz$\,\co the relation
$R_{\xy,\lph}\rp R_{\yz,\bt}$ is in $A$\per
\item[(iii)] For each
$\lph<\kai\xy$ and each $\bt<\kai\yz$\,\co
\[R_{ \xy, \alpha} \rp R_{ \yz,\bt}=\tbigcup\{R_{\xz,\g} :
\hs \xz \g \seq \vphi \xy\mo[ \ks\xy \alpha\scir\hs \yz \bt]\}\per\]
\item[(iv)]$ \h\xz\seq\vphi\xy\mo[K_{\xy}\scir\h\yz]$ and
$\vphih\xy\rp\vphih\yz=\vphih\xz$\,\co where $\vphih\xy$ and
$\vphih\xz$ are the mappings induced by $\vphi\xy$ and $\vphi\xz$ on
the quotient of $\G\wx$ modulo the normal subgroup
$\vphi\xy\mo[K_{\xy}\scir\h\yz]$\,\co while $\vphih\yz$ is the
isomorphism induced by $\vphi\yz$ on the quotient of $\G\wy$ modulo
the normal subgroup $K_{\xy}\scir\h\yz$. \end{enumerate}
Consequently\co the set $A$ is closed under relational composition
if and only if \textnormal{(iv)} holds for all pairs  $\pair x y$
and $\pair y z$ in $\mc E$\po
\end{theorem}
\begin{proof}

Let $\p 0$ be the normal subgroup of $\G y$ generated by $K_{\xy}$
and $\h\yz$,  \begin{align*}  \p 0
&=K_{\xy}\scir\h\yz.\tag{1}\label{Eq:cot1}\\ \intertext{Choose a
coset system
 $\langle \p \xi:\xi<\mu\rangle$  for $\p 0$ in $\G y$, and write}
\m \xi &=\vphi\xy\mo[\p\xi\,]\tag{2}\label{Eq:cot2}\\
\intertext{for $\xi<\mu$. The isomorphism properties of $\vphi\xy$
imply that}  \m 0=
\vphi\xy\mo[\p0]&=\vphi\xy\mo[K_{\xy}\scir\h\yz]\tag{3}\label{Eq:cot3.0}\\
\intertext{is a normal subgroup of $\G x$   that includes $\h\xy$
(the inverse image of  $K_{\xy}$ under $\vphi\xy$), and that
 the sequence $\langle \m \xi:\xi<\mu\rangle$ is a coset system
for $\m 0$ in $\G x$.  Moreover, the isomorphism $\vphi\xy$ induces
a quotient isomorphism $\vphih\xy$ from $\gxm$ to $\gyp$ that maps
$\m \xi$ to $ \p \xi$\ for each $\xi<\mu$. Similarly,  write}  \n
\xi &=\vphi\yz[\p\xi\,]\tag{4}\label{Eq:cot3}\end{align*} for
$\xi<\mu$, and observe that $\n 0$ is a normal subgroup of $\G z$
that includes $K_{\yz}$ (the image of $\h\yz$ under $\vphi\yz$), and
that the sequence  $\langle \n \xi:\xi<\mu\rangle$ is a coset system
for $\n 0$ in $\G z$. Moreover, the isomorphism $\vphi\yz$ induces a
quotient isomorphism $\vphih\yz$ from $\gyp$ to $\gzn$ that maps $\p
\xi$ to $ \n \xi$\ for each $\xi<\mu$.

Since $\p 0$ is a union of cosets of $K_{\xy}$, each coset of $\p 0$
is a union of cosets of $K_{\xy}$.  Thus, there is a partition
$\langle \Ga_\xi:\xi<\mu\rangle$ of $\kai\xy$ such that
\begin{align*}\p \xi&=\tbigcup \{\ks\xy \la:\la\in\Ga_\xi\}\tag{5}\label{Eq:cot4}\\
\intertext{for each $\xi<\mu$. Apply $\vphi\xy\mo$ to both sides of
\eqref{Eq:cot4}, and use the distributivity of inverse images,
together with  \eqref{Eq:cot2}, to obtain} \m \xi&=\tbigcup \{\hs\xy
\la:\la\in\Ga_\xi\}\tag{6}\label{Eq:cot5}\\ \intertext{for
$\xi<\mu$\per  Carry out   a completely analogous argument with
$\h\yz$ in place of $K_{\xy}$ to obtain a partition $\langle
\De_\xi:\xi<\mu\rangle$ of $\kai\yz$ such that} \p \xi&=\tbigcup
\{\hs\yz \la:\la\in\De_\xi\}\tag{7}\label{Eq:cot6}\\
\intertext{for each $\xi<\mu$. Apply $\vphi\yz$ to both sides of
\eqref{Eq:cot6}, and use the distributivity of forward images, to
obtain} \n \xi&=\tbigcup \{\ks\yz
\la:\la\in\De_\xi\}\tag{8}\label{Eq:cot7}
\end{align*}
for $\xi<\mu$.

It is well known from group theory that the intersection of the two
normal subgroups $K_{\xy}$ and $\h \yz$ in $\cs Gy$ is again a
normal subgroup in $\cs Gy$, and that a coset system for this
intersection is just the system of intersecting cosets,
  \begin{equation*}\langle \ks\xy
\la\cap\hs\yz \chi:\x<\mu\text{ and }\la\in  \subs \Ga \xi\text{ and
}\chi\in\subs\De\xi\,\rangle\per\tag{9}\label{Eq:cot8}\end{equation*}
In particular, the
 cosets in \eqref{Eq:cot8} are non-empty and mutually disjoint.
Moreover,\begin{multline*}
        \p \xi =\p\xi\cap\p\xi   =(\tbigcup \{\ks\xy \la:\la\in\Ga_\xi\})\cap
        (\tbigcup
\{\hs\yz \la:\la\in\De_\xi\})\\
         = \tbigcup\{
\ks\xy \la\cap\hs\yz \chi:\la\in  \subs \Ga \xi\text{ and
}\chi\in\subs\De\xi\}
         \end{multline*}
for each $\x<\mu$, by \eqref{Eq:cot4}, \eqref{Eq:cot6}, and the
distributivity of intersection. The composition of the relations
\begin{equation*}\tag{10}\label{Eq:cot9} (H_{\xy,\dlt}
\scir H_{\xy,\lph}\mo) \times (K_{\xy,\dlt} \cap H_{\yz,\ch})
\quad\text{and}\quad(K_{\xy,\zt} \cap H_{\yz,\vth}) \times
(K_{\yz,\vth} \scir K_{\yz,\bt})
\end{equation*}
 is  empty if either $\dlt\not=\zt$ or $\ch\not=\vth$, and it is
\begin{equation*}\tag{11}\label{Eq:cot10}
(H_{\xy,\dlt} \scir H_{\xy,\lph}\mo) \times (K_{\yz,\ch} \scir
K_{\yz,\bt})
\end{equation*}
 when $\dlt=\zt$ and $\ch=\vth$.

 If $\rh<\mu$ and $\dlt\in\subs\Ga\rh$, then
\begin{equation*}
K_{\xy,\dlt} = K_{\xy,\dlt} \cap P_\rh = K_{\xy,\dlt}\cap
(\mmmyspace \tbigcup_{\ch\in\subs\De\rh} H_{\yz,\ch} ) =
\tbigcup_{\ch\in\subs\De\rh}  (K_{\xy,\dlt}\cap
H_{\yz,\ch})\comma\tag{12}\label{Eq:cot11}
\end{equation*}
by \eqref{Eq:cot4}, \eqref{Eq:cot6}, and the distributivity of
intersection.
 A completely  analogous argument shows that \begin{equation*}
H_{\yz,\vth} = \tbigcup_{\zt \in\subs\Ga\rh} K_{\xy,\zt}\cap
H_{\yz,\vth}\tag{13}\label{Eq:cot13.0}
\end{equation*}
 for $\rh<\mu$ and $\vth\in\subs\De\gm$.

 Without any special assumptions on $A$, we now prove
that\begin{multline*}   \ks \xy\lph\seq\subs P\xi\ \text{ and }\ \hs
\yz\bt\seq\subs P\sgm\ \text{ implies }\\ R_{{\xy},{\lph}} \rp
R_{{\yz},{\bt}} = \tbigcup_{\rh<\mu}\m\rh\times(\n\rh\scir N_\x
\scir N_\sgm).\tag{14}\label{Eq:cot12}\end{multline*}

Successive  transformations of terms lead to the following values
for $R_{{\xy},{\lph}} \rp R_{{\yz},{\bt}}$.
\[\big[\tbigcup_{ \dlt<\kai\xy}
   H_{\xy,\dlt} \times (K_{\xy,\dlt} \scir K_{\xy,\lph}) \big]
   \relprod \big[\tbigcup_{ \vth<\kai\yz} H_{\yz,\vth} \times
   (K_{\yz,\vth} \scir K_{\yz,\bt}) \big]\comma\] by the
   definitions of $R_{{\xy},{\lph}}$ and $R_{{\yz},{\bt}}$.
\[\big[\tbigcup_{ \rh < \mu}\tbigcup_{\dlt \in\subs\Ga\rh}
   H_{\xy,\dlt} \times (K_{\xy,\dlt} \scir K_{\xy,\lph}) \big]
   \relprod \big[\tbigcup_{ \gm <
   \mu}\tbigcup_{\vth\in\subs\De\gm} H_{\yz,\vth} \times
   (K_{\yz,\vth} \scir K_{\yz,\bt}) \big]\comma\] because the sets
   $\Ga_\rho$ (for $\rho<\mu$) and $\De_\gm$ (for $\gm<\mu$)
   partition $\kai\xy$ and $\kai\yz$ respectively.
\[ \big[ \tbigcup_{\rh,\dlt} (H_{\xy,\dlt} \scir
   H_{\xy,\lph}\mo) \times K_{\xy,\dlt} \big]\relprod \big[
   \tbigcup_{\gm,\vth} H_{\yz,\vth} \times (K_{\yz,\vth}\scir
   K_{\yz,\bt}) \big]\comma \] by Lemma~\ref{L:PQnormsg}(i)\per
\begin{multline*} \big[ \tbigcup_{\rh,\dlt} (H_{\xy,\dlt}\scir
   H_{\xy,\lph}\mo)\times (\,\tbigcup_{\ch\in\subs\De\rh}
   K_{\xy,\dlt}\cap H_{\yz,\ch})\big]\relprod  \\ \big[
   \tbigcup_{\gm,\vth} (\,\tbigcup_{\zt\in\subs\Ga\gm}
   K_{\xy,\zt}\cap H_{\yz,\vth}) \times (K_{\yz,\vth}\scir
   K_{\yz,\bt})\big]\comma
\end{multline*}
by \eqref{Eq:cot11} and \eqref{Eq:cot13.0}\per
\begin{multline*}
\big[ \tbigcup_{\rh,\dlt,\ch} (H_{\xy,\dlt}\scir H_{\xy,\lph}\mo)
\times (K_{\xy,\dlt}\cap H_{\yz,\ch})\big]\relprod\\ \big[
\tbigcup_{\gm,\vth,\zt} (K_{\xy,\zt}\cap H_{\yz,\vth}) \times
(K_{\yz,\vth}\scir K_{\yz,\bt})\big]\comma
\end{multline*}
by the distributivity of Cartesian multiplication.
\begin{multline*}\tbigcup_{\rh,\dlt,\ch,\gm,\vth,\zt}
   [ (H_{\xy,\dlt}\scir H_{\xy,\lph}\mo) \times (K_{\xy,\dlt}\cap
   H_{\yz,\ch})]\relprod\\ [ (K_{\xy,\zt}\cap H_{\yz,\vth}) \times
   (K_{\yz,\vth} \scir K_{\yz,\bt}) ]\comma
\end{multline*}
by the distributivity of relational composition.  The relational
composition inside the brackets of the preceding union is precisely
of the form \eqref{Eq:cot9}. Apply  the conclusions of
\eqref{Eq:cot9} about this relative product, and in particular use
\eqref{Eq:cot10}, the distributivity of complex group composition,
\eqref{Eq:cot5}, and \eqref{Eq:cot7} to obtain
\begin{align*}
R_{{\xy},{\lph}} \relprod R_{{\yz},{\bt}} =& \tbigcup_{\rh,\dlt,\ch}
(H_{\xy,\dlt} \scir H_{\xy,\lph}\mo) \times (K_{\yz,\ch} \scir
K_{\yz,\bt})  \\ =& \tbigcup_\rh [ (\,\tbigcup_\dlt
H_{\xy,\dlt})\scir H_{\xy,\lph}\mo ] \times [ (\,\tbigcup_\ch
K_{\yz,\ch}) \scir K_{\yz,\bt}]  \\ =& \tbigcup_\rh (M_\rh \scir
H_{\xy,\lph}\mo) \times (N_\rh\scir K_{\yz,\bt})  \\ =&
\tbigcup_\rh (M_\rh \scir M_\x\mo) \times (N_\rh\scir N_\sgm) \\
=& \tbigcup_\rh M_\rh \times (N_\rh\scir N_\x \scir N_\sgm)\per
\end{align*}    The fourth step is a consequence
of the following well-known   property of cosets:   $K_{\yz}$ is a
normal subgroup of $\cs N\rh$, so if the coset $\ks\yz\bt$ is
included in the coset $\cs N\sgm$, then \[\cs N\rh\scir\ks\yz\bt=\cs
N\rh\scir\cs N\sgm\] (and similarly for the passage from $\cs
M\rh\scir\hs\xy\alpha\mo$ to $\cs M\rh\scir\cs M\xi\mo$).
 This completes the proof of \eqref{Eq:cot12}.

Assume
\begin{equation*}\tag{15}\label{Eq:cot14}
\h\xz\seq\m 0=\cs\vp\xy\mo[K_{\xy}\scir\h\yz]\per
\end{equation*}  There must then be a partition $\langle\subs\Ps\rh:\rh<\mu\rangle$ of $\kai\xz$ such
that
\begin{equation*}\tag{16}\label{Eq:cot15}
\m \rh = \tbigcup_{\dlt\in\subs\Ps\rh}\hs\xz \dlt
\end{equation*} for each $\rh <\mu$.
 Apply $\vphi\xz$ to both sides of \eqref{Eq:cot15}, and use the
correspondence
\begin{equation*} \tag{17}\label{Eq:cot16}
  \vphi\xz\myshortspace[\hs\xz \dlt]=\ks\xz\dlt
\end{equation*}
  and the distributivity of forward images, to see that  the   function $\vphih\xz$
defined by
\begin{equation*}\tag{18}\label{Eq:cot19.0}
\vphih\xz(\m\rh)=\cs\vp\xz[\m\rh]=\cs\vp\xz[\mmmyspace\tbigcup_{\delta\in
\cs\Ps\rh}\hs\xz\dlt]=\tbigcup_{\dlt\in\cs\Ps\rh}\ks\xz\dlt
\end{equation*}
  is a well-defined  isomorphism on the quotient group $\G x/\m
0$. A simple computation using \eqref{Eq:cot2} and \eqref{Eq:cot3}
shows that
\begin{equation*}\tag{19}\label{Eq:cot13}
(\cs{\hat\vp}\xy\rp\cs{\hat\vp}\yz)(\m\rh)=\cs{\hat\vp}\yz\bigl(\cs{\hat\vp}\xy(\m\rh))
=\cs{\hat\vp}\yz(\p\rh)=\n \rh
\end{equation*} for each $\rh<\mu$.
Combine \eqref{Eq:cot13} and \eqref{Eq:cot19.0} to conclude that the
following three conditions are equivalent:
\begin{equation*}
 \vphih\xy\rp\vphih\yz =\vphih\xz\comma \quad \vphih\xz(\m
\rh)=\n\rh\text{ for all $\rh$}\comma\quad \n
\rh=\tbigcup_{\dlt\in\subs\Ps\rh}\ks\xz\dlt\text{ for all
$\rh$}\per\tag{20}\label{Eq:cot18}
\end{equation*}

Turn now to the task of establishing the equivalences in the
theorem. Assume (iv), with the goal of deriving (iii). Fix
$\alpha<\kai\xy$ and $\bt<\kai\yz$, and choose $\xi,\sgm<\mu$ so
that $\ks \xy\lph\seq\subs P\xi$ and $\hs \yz\bt\seq\subs P\sgm$.
The
 multiplication rules for cosets  imply that
\begin{equation*}
 \p \x\scir\p\sigma=\ks\xy \alpha\scir\hs\yz\bt\per
\tag{21}\label{Eq:cot19}\end{equation*}  Choose $\pi$ so that
\begin{align*}\p\xi\scir\p \sgm&=\p
\pi\comma\tag{22}\label{Eq:cot20}\\ \intertext{and use the
isomorphism properties of $\vphih\yz$, together with
\eqref{Eq:cot3}, to obtain}
 \n\xi\scir\n \sgm&=\n \pi\per\tag{23}\label{Eq:cot21}
\end{align*}
Compute:
\begin{align*}
R_{{\xy},{\lph}} \rp R_{{\yz},{\bt}}&=
\tbigcup_{\rh<\mu}\m\rh\times(\n\rh\scir\n\pi)\\ &= \tbigcup_\rh
(\mmmyspace\tbigcup_{\dlt\in\subs\Ps\rh} H_{\xz,\dlt})
\times(N_\rh\scir N_\pi) \\ &= \tbigcup_{\rh,\dlt}
H_{\xz,\dlt}\times (N_\rh\scir N_\pi) \\ &= \tbigcup_{\rh,\dlt}
H_{\xz,\dlt}\times (K_{\xz,\dlt} \scir N_\pi)
\\ &= \tbigcup_{\rh,\dlt} H_{\xz,\dlt}\times
   [ K_{\xz,\dlt}\scir ( \tbigcup_{\gm \in\subs\Ps\pi}
   K_{\xz,\gm})]\\ &= \tbigcup_{\gm } \tbigcup_{\rh,\dlt}
   H_{\xz,\dlt}\times (K_{\xz,\dlt}\scir K_{\xz,\gm})\\ &=
   \tbigcup_{\gm } \tbigcup_{\dlt<\kai\xz} H_{\xz,\dlt}\times
   (K_{\xz,\dlt}\scir K_{\xz,\gm}) \\ &= \tbigcup_{\gm }
   R_{{\xz},{\gm}}\per
\end{align*} The first equality uses \eqref{Eq:cot12} and
\eqref{Eq:cot21}, the second uses \eqref{Eq:cot15}, the third uses
the distributivity of Cartesian multiplication, the fourth uses the
multiplication rules for cosets and the inclusion of $\ks\xz\dlt$ in
$\n\rh$ (because $\dlt$ is in $\cs\Ps\rh$, and (iv) holds, which
implies that the last equation in \eqref{Eq:cot18} holds), the fifth
uses the assumption of (iv), which implies that the last equation in
\eqref{Eq:cot18} holds with $\pi$ in place of $\rho$, the sixth uses
the distributivity of   complex group composition, the seventh uses
the fact that the sets $\Ps_\rh$ (for $\rh<\mu$) partition
$\kai\xz$, and the last uses the definition of $R_{\xz,\g}$.
Summarizing,
\begin{equation*}  R_{{\xy},{\lph}} \rp R_{{\yz},{\bt}}=\tbigcup_{\gm \in \subs \Ps\pi} R_{{\xz},{\gm}}\per\tag{24}\label{Eq:cot22}
\end{equation*}

In order to complete the derivation of (iii), it is necessary to
characterize the relations $R_{\xz,\g}$ such that
$\g\in\subs\Ps\pi$. First,
\begin{equation*}\m\pi=\vphi\xy\mo[\p\pi]=\vphi\xy\mo[\p\xi\scir\p\sgm]=\vphi\xy\mo[\ks\xy\alpha\scir\hs\yz\bt],
\tag{25}\label{Eq:cot23}
\end{equation*}
by \eqref{Eq:cot2} (with $\pi$ in place of $\rho$),
\eqref{Eq:cot20}, and \eqref{Eq:cot19}. Second,
\begin{equation*}\m\pi =
\vphi\xz\mo[\n\pi]=\vphi\xz\mo[\mmmyspace\tbigcup_{\g\in\subs\Ps\pi}\ks\xz
\dlt]=\mmmyspace\tbigcup_{\g\in\subs\Ps\pi}\vphi\xz\mo[\ks\xz
\dlt]=\tbigcup_{\g\in\subs\Ps\pi}\hs\xz
\dlt\comma\tag{26}\label{Eq:cot24}
\end{equation*}
by the assumption in (iv), which implies that the second and third
equations in  \eqref{Eq:cot18} hold, the distributivity of inverse
images, and \eqref{Eq:cot16}. Combine \eqref{Eq:cot23} and
\eqref{Eq:cot24} to arrive at the equivalence of the three
conditions
\begin{equation*}
\hs\xz \g\seq\vphi\xy\mo[\ks\xy\alpha\scir\hs\yz\bt]\comma\qquad
\hs\xz \g\seq\m\pi\comma \qquad \gm\in \subs\Ps\pi\per
\end{equation*}
The equivalence of the first and last formulas permits us to rewrite
\eqref{Eq:cot22} in the desired form of (iii):
\begin{equation*}R_{\xy, \alpha}\rp R_{\yz,\bt}
= \tbigcup\{R_{\xz,\g} : \g<\kai\xz\text{ and }\hs\xz \g\seq
\vphi\xy\mo[\ks\xy\alpha\scir\hs\yz\bt]\}\per
\end{equation*}

The implication from (iii) to (ii) follows from the definition of
 $A$ as the set of arbitrary unions of atomic relations (see Boolean Algebra Theorem~\ref{T:disj}),
 and the implication from (ii) to (i) is obvious.
Consider finally the implication from (i) to (iv). Under the
assumption of (i), the definition of $A$ implies that $R_{\xy,0}\rp
R_{\yz, 0}$ must be a union of atomic relations of the form
$R_{\xz,\zeta}$. In more detail, the inclusions and the equality
\[\ks\xy 0\seq \p 0,\qquad \hs\yz 0\seq\p 0,\qquad \n\rh\scir\n
0\scir\n 0=\n\rh
\]  (which hold by \eqref{Eq:cot4}, \eqref{Eq:cot6}, and the fact that $\n 0$ is the identity
coset), together with  \eqref{Eq:cot12}, imply  (without using the
hypothesis of (i)) that
\begin{equation*}\tag{27}\label{Eq:cot25}
R_{\xy, 0}\rp R_{\yz, 0}=\tbigcup_{\rh<\mu}\m \rh\times \n\rh\per
\end{equation*}
  The cosets $\m \rh$ and $\n \rh$ are  subsets of $\G x$ and  $\G
z$ respectively, for each $\rh<\mu$ (see \eqref{Eq:cot2} and
\eqref{Eq:cot3}), so the right-hand side---and therefore also the
left-hand side---of \eqref{Eq:cot25} must be a subset of the
rectangle $\gsq x z$.  The only atomic relations in $A$ that are not
disjoint from this rectangle are those of the  form $R_{\xz,\zeta}$.
Conclusion: there is a subset $\Ph$ of $\kai\xz$ such that
\begin{equation*}\tag{28}\label{Eq:cot26}
  R_{\xy, 0}\rp R_{\yz, 0}=\tbigcup_{\zeta\in\Ph}R_{\xz,\zeta}.
\end{equation*}

 The pair  $\pair {\e x\,}{\e z}$   belongs to the rectangle $\m
0\times \n 0$, and therefore also to the composition $R_{\xy,0}\rp
R_{\yz, 0}$, by \eqref{Eq:cot25}. The pair also belongs to the
rectangle
\[\hs\xz 0\times\ks\xz 0 = \h\xz\times K_{\xz}\comma \] and
therefore to the relation
\begin{equation*}\tag{29}\label{Eq:cot27}
  R_{\xz, 0}=\tbigcup_{\g<\kai\xz}\hs\xz \g\times\ks\xz \g\per
\end{equation*}
  The relations of the form $R_{\xz, \zeta}$ are pairwise disjoint,
and $R_{\xz, 0}$ is the only one  that contains the pair $\pair {\e
x\,}{\e z}$. In view of \eqref{Eq:cot26}, the only way  this can
happen is if $0$ is one of the indices in $\Ph$, so that
\begin{equation*}R_{\xz, 0}\seq R_{\xy, 0}\rp R_{\yz,0}
\tag{30}\label{Eq:cot28}\per\end{equation*}Use \eqref{Eq:cot27} and
\eqref{Eq:cot25} to rewrite \eqref{Eq:cot28} in the form
\begin{equation*}\tbigcup_{\g<\kai\xz}\hs\xz \g\times\ks\xz \g\seq
\tbigcup_{\rh<\mu}\m \rh\times \n\rh\per\tag{31}\label{Eq:cot28.0}
\end{equation*}

In view of Lemma~\ref{L:rect2}, the inclusion in \eqref{Eq:cot28.0}
implies that for every $\gm<\kai\xz$ there is a $\rh<\mu$ such that
\begin{equation*}\hs\xz \g\seq \m
\rh\qquad\text{and}\qquad\ks\xz \g\seq \n {\rh}\per\tag{32}
\label{Eq:cot29}\end{equation*} In particular, when $\g=0$, the
subgroup $\hs\xz 0$ is included in  $\m\rh$ for some $\rh<\mu$. This
inclusion forces the group identity element $e_x$ to belong to
$\m\rh$, since it belongs to  $\hs\xz 0$. The only coset of $\m 0$
that includes the group identity element is $\m 0$ itself, so
$\rh=0$, that is to say, \eqref{Eq:cot14} holds.

On the basis of \eqref{Eq:cot14} alone, we saw that a partition
$\langle\subs\Ps\rh:\rh<\mu\rangle$ of $\kai\xz$ exists for which
  \eqref{Eq:cot15} holds.  The derivations of  \eqref{Eq:cot15}--\eqref{Eq:cot18}
are based only on \eqref{Eq:cot13} and \eqref{Eq:cot14},  so these
statements hold in the present situation as well. It is evident from
\eqref{Eq:cot15} that the coset $\hs\xz \dlt$ is included in $\m
\rh$ for each
 $\dlt\in\subs\Ps\rh$, and therefore the  corresponding coset
$\ks\xz \dlt$ must be included in $\n\rh$ for each
$\dlt\in\subs\Ps\rh$ by \eqref{Eq:cot29}.
Thus,\begin{equation*}\tbigcup_{\dlt\in\subs\Ps\rh}\ks\xz\dlt\seq\n\rh\tag{33}\label{Eq:cot30}
\end{equation*}
 for each $\rh<\mu$.  The cosets $\ks\xz\dlt$ (for $\rh<\mu$ and
$\delta$ in $\cs\Ps\rh$) form a partition of $\G z$, as do
 the cosets $\n \rh$ (for $\rh<\mu$). These two facts force the
inclusion in \eqref{Eq:cot30} to be an equality. Use this equality
and the equivalence of the first and third formulas in
\eqref{Eq:cot18} to conclude that $\vphih\xy\rp\vphih\yz
=\vphih\xz$, as desired in (iv).

Turn now to the final assertion of the theorem.  If $A$ is closed
under relational composition, then (ii)   holds for all pairs $\pair
xz $ and $\pair yz$ in $\mc E$, so that (iv) must hold by the
equivalence of (ii) and (iv) proved above.

On the other hand, if (iv) holds for all pairs $\pair xy$ and $\pair
yz$ in $E$, then (ii) holds as well.  Combine this with
Lemma~\ref{L:emptycomp} to conclude that $A$ is closed under the
composition of any two of its atomic relations.  The elements in $A$
are just the various possible unions of atomic relations, and the
composition of two such unions is  a union of  compositions of
atomic relations, by the distributivity of relational composition,
and hence a union of elements in $A$. Since $A$ is closed under
arbitrary unions, it follows that the composition of any two
elements in $A$ is again an element in $A$, as was to be shown.
\end{proof}

 Notice that part (iii) of the previous theorem provides a concrete
way of computing the composition of any two relations
$R_{\xy,\alpha}$ and $R_{\yz,\bt}$ in terms of the structure of the
quotient group $\G y/(K_{\xy}\scir\h\yz)$, the mapping $\vphi\xy$,
and the cosets $\hs\xz\gm$. One first computes the complex product
$\ks\xy\lph\scir\hs\yz\bt$, and then the inverse image of this
complex under $\vphi\xy\mo$.  This inverse image is a union of
cosets  $\hs\xz\gm$, and one computes the set $\Ga$ of indices $\gm$
for which the corresponding cosets are part of this union.

It is natural to ask whether, in condition (i) of the Composition
Theorem, one can replace the condition that $R_{\xy, 0} \rp R_{\yz,
0}$ be in $A$ by the condition that $R_{\xy, \alpha} \rp R_{\yz,
\bt}$ be in $A$ for some $\alpha<\kai \xy$ and some $\bt<\kai\yx$.
It turns out that an additional hypothesis is needed for this to be
true. We will return to this matter at the end of the next section.

\begin{lemma}\label{L:domains} Suppose
$\pair x y $ and $\pair y z $ are in $\mc E$\per If $\vphi\yx
=\vphi\xy\mo$\,\co then
\begin{gather*}\h\xz\seq\vphi\xy\mo[K_{\xy}\scir\h\yz]\quad\text{and}\quad
\h\yz\seq\vphi\yx\mo[K_{\yx}\scir\h\xz]\\ \intertext{together imply}
\vphi\xy [\h\xy\scir\h\xz]=K_{\xy}\scir\h\yz=\h\yx\scir\h\yz.
\end{gather*}\end{lemma}
\begin{proof}
The function  $\vphi\xy$  is  an isomorphism from $\G x/\h \xy$ onto
$\G y/K_{\xy}$\co and the complex product $K_{\xy}\scir\h\yz$ is a
normal subgroup of $\G y$ that  includes $K_{\xy}$. Therefore, the
inverse image
\begin{equation*}\label{Eq:domains1}\tag{1}\vphi\xy\mo[K_{\xy}\scir\h\yz]
\end{equation*}
is a normal subgroup of  $\G x$  that includes
$\vphi\xy\mo[K_{\xy}]=\h\xy$.    By assumption, \eqref{Eq:domains1}
also includes $\h\xz$, so
\begin{align*} \h\xy\scir\h\xz&\seq
\vphi\xy\mo[K_{\xy}\scir\h\yz].\\ \intertext{Apply $\vphi\xy$ to
both sides of  this inclusion to obtain}
\vphi\xy[\h\xy\scir\h\xz]&\seq
K_{\xy}\scir\h\yz.\label{Eq:domains2}\tag{2}\\ \intertext{The same
argument with the roles of $x$ and $y$ reversed yields}
 \vphi\yx[\h\yx\scir\h\yz]&\seq
K_{\yx}\scir\h\xz.\label{Eq:domains3}\tag{3}
\end{align*}
The assumption $\vphi\yx=\vphi\xy\mo$ implies that
\begin{equation*}\label{Eq:domains5}\tag{4}\h\yx=K_{\xy}\qquad
\text{and}\qquad K_{\yx}=\h\xy\end{equation*}
 (see Convention~\ref{Co:convention}).  Use these three
equalities to rewrite the inclusion in  \eqref{Eq:domains3} as
\begin{equation*}\tag{5}\label{Eq:domains6}\vphi\xy\mo[K_{\xy}\scir\h\yz]\seq
\h\xy\scir\h\xz.
\end{equation*} Combine \eqref{Eq:domains6} with \eqref{Eq:domains2} and \eqref{Eq:domains5} to  arrive
at the desired conclusion.
\end{proof}

\begin{lemma}\label{L:sfimage} If $\vphi\yx
=\vphi\xy\mo$ for all pairs $\pair xy$ in $\mc E$\comma and if
\[\vphi\xy [\h\xy\scir\h\xz]=K_{\xy}\scir\h\yz\] for all   $\pair xy$ and $\pair yz$ in $\mc E$\comma
then
\begin{equation*}
\vphi\yz [K_{\xy}\scir\h\yz]=K_{\xz}\scir
K_{\yz}\qquad\text{and}\qquad \vphi\xz[\h\xy\scir\h\xz]=K_{\xz}\scir
K_{\yz}\per
\end{equation*}
 \end{lemma}
\begin{proof}    To derive the
first  equation, observe that if the pairs $\pair xy$ and $\pair yz$
are in $\mc E$ then so are the pairs $\pair xz$ and $\pair zx$. Use
the assumption that $\vphi\yx =\vphi\xy\mo$,
Convention~\ref{Co:convention}, the commutativity of normal
subgroups, and the hypotheses of the lemma (with $y$, $z$, and $x$
in place of $x$, $y$, and $z$ respectively) to arrive at
\[\vphi\yz[\kh]=\vphi\yz[\h\yx\scir\h\yz]
=\vphi\yz[\h\yz\scir\h\yx]=K_{\yz}\scir\h\zx=K_{\yz}\scir
K_{\xz}\per\] An entirely analogous argument yields
\[\vphi\xz[\hh]=\vphi\xz[\h\xz\scir\h\xy]=K_{\xz}\scir\h\zy=K_{\xz}\scir K_{\yz}\per
\]
\end{proof}

\begin{theorem}[Image Theorem]\label{T:domainofmap}  If $A$ is
closed under converse and composition\co then
\begin{gather*}
\vphi\xy [\h\xy\scir\h\xz]=K_{\xy}\scir\h\yz\comma\qquad \vphi\yz
[K_{\xy}\scir\h\yz]=K_{\xz}\scir K_{\yz}\comma\\
\vphi\xz[\h\xy\scir\h\xz]=K_{\xz}\scir K_{\yz}
\end{gather*}
for all $\pair x y$ and $\pair y z$ in $\mc E$\po In other words,
\begin{enumerate}
\item[(i)] $\vphih\xy$ maps  $\G x/(\h\xy\scir\h\xz)$
isomorphically to  $\G y/(K_{\xy}\scir\h\yz)$\co
\item[(ii)]
$\vphih\yz$ maps  $\G y/(K_{\xy}\scir\h\yz)$ isomorphically to $\G
z/(K_{\xz}\scir K_{\yz})$\co
\item[(iii)]
$\vphih\xz$ maps  $\G x/(\h\xy\scir\h\xz)$ isomorphically to $\G
z/(K_{\xz}\scir K_{\yz})$\po
\end{enumerate}
\end{theorem}
\begin{proof}  Assume $A$ is closed under converse and composition, and consider pairs $\pair xy$ and $\pair yz$
in $\mc E$.   The assumed closure of $A$ under converse means that
part (iii) of the  Converse Theorem~\ref{T:convthm1} may be applied
to obtain
\begin{equation*}\tag{1}\label{Eq:image1}\vphi\yx =\vphi\xy\mo\per
\end{equation*}  The assumed closure of $A$ under composition means that part (iv) of the Composition Theorem~\ref{T:compthm} may
be applied to obtain
\begin{equation*}
\h\xz\seq\vphi\xy\mo[K_{\xy}\scir\h\yz]\qquad\text{and}\qquad \h\yz
\seq\vphi\yx\mo[K_{\yx}\scir\h\xz]\per
\end{equation*}   Invoke Lemma~\ref{L:domains} to obtain
\begin{equation*}\tag{2}\label{Eq:image2}
\vphi\xy [\h\xy\scir\h\xz]=K_{\xy}\scir\h\yz\per
\end{equation*} This argument establishes the first equation for all pairs $\pair xy$ and $\pair yz$ in $\mc
E$.   Use Lemma~\ref{L:sfimage} to obtain the second and third
equations.

The mappings $\vphih\xy$ and $\vphih\xz$ are defined to be the
isomorphisms induced on the quotient group of $\G x$ modulo
$\vphi\xy\mo[K_{\xy}\scir\h\yz]$ by the isomorphisms $\vphi\xy$ and
$\vphi\xz$ respectively, and  $\vphih\yz$ is defined to be the
isomorphism  induced on the quotient group of $\G y$ modulo $
K_{\xy}\scir\h\yz$ by the isomorphism $\vphi\yz$, by part (iv) of
the Composition Theorem~\ref{T:compthm}. In view of the preceding
proof, this immediately gives assertions (i)--(iii) of the theorem.
\end{proof}

\section{Group frames}

In the preceding section,   necessary and sufficient conditions are
given for a Boolean algebra $A$ of binary relations constructed from
a group pair $\mc F$ to  contain the identity relation and to be
closed under the operations of relational converse and composition.
In each case, one of these conditions is formulated strictly in
terms of the  quotient isomorphisms.  It is natural to single out
the group pairs that satisfy these quotient isomorphism conditions,
because precisely these group pairs lead to algebras of binary
relations and in fact to measurable relation algebras.

\begin{definition}\label{D:cosfra} A \textit{group} \textit{frame}  is a group pair
\[\mc F=(\langle \G x:x\in
I\,\rangle\smbcomma\langle\vph_\xy:\pair x y\in \mc E\,\rangle)\]
satisfying the following  \textit{frame conditions} for all pairs
$\pair xy$ and $\pair yz$ in $\mc E$\per
\begin{enumerate}
\item[(i)]
$\vphi {xx}$ is the identity automorphism of $\G x/\{\e x\}$ for all
$x$\per
\item [(ii)]
$\vphi \yx=\vphi \xy\mo$\per
\item[(iii)] $\vphi \xy[\h \xy\scir\h\xz] =K_{\xy}\scir\h \yz$\per
\item[(iv)] $\vphih \xy\rp\vphih \yz=\vphih \xz$\po
\end{enumerate}\qed\end{definition}

Given a group frame \ $\mc F$, let $A$  be the collection of all
possible unions of relations of the form $R_{\xy, \alpha}$ for
$\pair x y$ in $\mc E$ and $\alpha <\kai\xy$.  Call $A$ the set of
\textit{frame relations} constructed from $\mc F$.

\begin{theorem}[Group Frame Theorem]\label{T:closed} If $\mc F$ is a group
frame\co then the set  of frame relations
constructed from $\mc F$ is the universe of a complete\co atomic\co
measurable set relation algebra   with   base set and unit
\[U=\tbigcup  \{\G x:x\in I\}\qquad\text{and}\qquad E=\tbigcup\{\cs
Gx\times\cs Gy:\pair xy\in\mc E\}\] respectively\per The atoms in
this algebra are the relations of the form $R_{\xy, \alpha}$\comma
and the subidentity atoms are the relations of the form $R_{\xx,
0}$\per The measure of $R_{\xx, 0}$ is just the cardinality of the
group $\cs Gx$\per
\end{theorem}
 \begin{proof} Let $A$ be the set of frame relations constructed from $\mc F$.
 This set  is the universe of a complete and atomic Boolean
algebra of binary relations with base set $U$ and unit $E$, and its
atoms are the relations of the form $R_{\xy,\alpha}$\comma by
Boolean Algebra Theorem~\ref{T:disj}. The identity relation $\id U$
is in $A$, and the subidentity atoms are the relations of the form
$R_{\xx, 0}$, by Theorem~\ref{T:disj}, Identity
Theorem~\ref{T:identthm1}, and frame condition (i). The closure of
$A$ under the operations of converse and composition follows from
Converse Theorem~\ref{T:convthm1}, Composition
Theorem~\ref{T:compthm}, and frame conditions (ii)--(iv).

The measure of a subidentity atom $R_{\xx, 0}$ is, by definition,
the number of non-zero functional atoms below the square \[R_{\xx,
0}\rp E\rp R_{\xx, 0}=\cs Gx\times\cs Gx\per\] These non-zero
functional atoms are just the relations $R_{\xx, \alpha}$ for
$\alpha< \kai\xx$, that is to say, they are just the Cayley
representations of the elements in $\cs Gx$, by Partition
Lemma~\ref{L:i-vi}. Consequently, there are as many of them as there
are elements in $\cs Gx$.
\end{proof}

The  theorem justifies the following definition.

\begin{definition}\label{D:gradef} Suppose that $\mc F$ is a group frame\po The
 set \ra\ constructed from $\mc F$ in   Group Frame Theorem~\ref{T:closed} is
 called the (\textit{full}) \textit{\ggra} on $\mc F$ and is
 denoted by $\cra G F$ \opar and its universe by  $\craset G
 F$\cpar\po A \textit{general group relation algebra} is defined
 to be an algebra that is embeddable into a full group relation
 algebra\po \qed
\end{definition}

The task of verifying  that a given group pair satisfies the frame
conditions, and therefore yields a full group relation algebra, that
is to say, it yields an example of a measurable relation algebra,
can be complicated and tedious.  Fortunately, a few simplifications
are possible. To describe them, it is helpful to assume that the
group index set $I$ is linearly ordered, say by a relation $\,<\,$.
Roughly speaking, under the assumption of condition (i), condition
(ii) holds in general just in case it holds for each pair $\pair xy$
in $\mc E$ with $x<y$. In other words, under the assumption of (i),
it is not necessary to check condition (ii) for the case $x=y$, nor
is it necessary to check both cases $\pair xy$ and $\pair yx$ when
$x\neq y$. Also, under the assumption of condition (i) and the
modified form of condition (ii) just described, conditions (iii) and
(iv) will hold in general if they hold for all pairs $\pair xy$ and
$\pair yx$ in $\mc E$ with $x<y<z$. In other words, under the
assumption of
 (i) and the modified  (ii), it is not necessary to check
conditions (iii) and (iv) in any case in which at least two of the
three indices $x$, $y$, and $z$ are equal, nor is it necessary to
check all six   permutations of an appropriate triple $\trip xyz$ of
distinct indices.  Here is the precise formulation of the theorem.

\begin{theorem}\label{T:simpfr} A group pair $\mc F$   is a group
frame if and only if the following four conditions are satisfied\per
\begin{enumerate}
\item[(i)]
$\vphi {xx}$ is the identity automorphism of $\G x/\{\e x\}$ for
every $\wx$ in $I$\per
\item [(ii)]
$\vphi \yx=\vphi \xy\mo$ for every pair $ \pair x y$ in $\mc E$ with
$x<y$\per
\item[(iii)] $\vphi \xy[\h
\xy\scir\h \xz]=K_{\xy}\scir\h\yz$ and $\vphi \yz[K_{\xy}\scir\h
\yz]=K_{\xz}\scir K_{\yz}$  for all pairs $\pair xy$ and $\pair yz$
in $\mc E$ with $x<y<z$\per
\item[(iv)] $\vphih \xy\rp\vphih \yz=\vphih \xz$ for all pairs $\pair xy$ and $\pair yz$ in
$\mc E$ with $x<y<z$\per
\end{enumerate}
\end{theorem}
\begin{proof}By its very definition, a frame
must satisfy conditions (i)--(iv) of the theorem.  To establish the
reverse implication, suppose that
\[\mc F=(\langle \G x:x\in
I\,\rangle\smbcomma\langle\vph_\xy:\pair x y\in \mc E\,\rangle)\] is
a group pair satisfying conditions (i)--(iv) of the theorem. It must
be shown that the four frame conditions  hold. Obviously,
 the first frame condition holds,  since it coincides with
condition (i) of the theorem. To verify the second frame condition,
assume that $\pair xy$ is a pair in $\mc E$.   If $\wx=\wy$, then
$\vphi\xy$ is the identity automorphism, by condition (i) of the
theorem, and therefore
\begin{equation*}\tag{1}\label{Eq:sf1}
  \vphi\yx=\vphi\xy\mo\per
\end{equation*} If $x<y$, then \eqref{Eq:sf1} holds, by condition (i)
of the theorem.  If $y<x$, then $\vphi\xy=\vphi\yx\mo$, by condition
(i) of the theorem (with the roles of $x$ and $y$ reversed), so
\eqref{Eq:sf1} must also hold.

Turn now to the task of verifying the last two frame conditions.
Assume that $\pair xy$ and $\pair yz$ are pairs in $\mc E$, and
consider first the case when $x=y$.  The mapping $\vphi\xy$ is then
the identity automorphism of $\G x/\{e_x\} $, by condition (i), so
that
\begin{equation*}%
  \h \xy=\h\xx=\{\cs ex\}=K_{\xx}=K_{\xy}\comma\qquad
\h\xz=\h\yz\comma\qquad K_{\xz}=K_{\yz}\comma
\end{equation*}
and therefore
\begin{equation*}
\h\xy\scir\h\xz=\h\xz\comma\quad
K_{\xy}\scir\h\yz=\h\yz=\h\xz\comma\quad K_{\xz}\scir
K_{\yz}=K_{\yz}\scir K_{\yz}=K_{\yz}\per
\end{equation*}
It follows that
\begin{gather*}
\vphi\xy[\h\xy\scir\h\xz]=\vphi\xx[\h\xz]=\h\xz=K_{\xy}\scir\h\yz\\
\intertext{and}
\vphi\yz[K_{\xy}\scir\h\yz]=\vphi\yz[\h\yz]=K_{\yz}=K_{\xz}\scir
K_{\yz}\per
\end{gather*}
For the same reasons, the isomorphism $\vphih\xy$ induced by
$\vphi\xy$ on the quotient group $\G x/(\h\xy\scir\h\xz) $ must
coincide with the identity automorphism of $\G x/\h\xz$, the
isomorphism $\vphih\yz$ induced by $\vphi\yz$ on the quotient group
$\G y/(K_{\xy}\scir\h\yz) $ must coincide with $\vphi\yz$, and the
isomorphism $\vphih\xz$ induced by $\vphi\xz$ on the quotient group
$\G x/(\h\xy\scir\h\xz) $ must coincide with the isomorphism
$\vphi\xz$\per Consequently,
\[\vphih\xy\rp\vphih\yz=\vphi\yz=\vphi\xz=\vphih\xz\per
\]

The case when $y=z$ is treated in a completely symmetric fashion.
Consider, next, the case when $x=z$\per  The mapping $\vphi\xz$ is
then the identity automorphism of $\G x/\{e_x\} $, by condition (i),
so that
\begin{equation*}
  \h \xz=\h\xx=\{\cs ex\}=K_{\xx}=K_{\xz}\comma\qquad
  \h\yz=\h\yx\comma\qquad K_{\yz}=K_{\yx}\comma
\end{equation*}
and therefore
\begin{gather*}
\h\xy\scir\h\xz=\h\xy\comma\quad
K_{\xy}\scir\h\yz=K_{\xy}\scir\h\yx=K_{\xy}\scir
  K_{\xy} =\\
  K_{\xy}=\h\yx=\h\yz\comma\quad K_{\xz}\scir K_{\yz}=K_{\yz}=K_{\yx}\per
\end{gather*}
In the second string of equations, use is being made of the fact
that the second frame condition holds, and therefore $K_{\xy}=\h\yx$
(see the remark preceding Theorem~\ref{T:convthm1}). It follows that
\begin{gather*}
\vphi\xy[\h\xy\scir\h\xz]=\vphi\xy[\h\xy]=K_{\xy}=K_{\xy}\scir\h\yz\\
\intertext{and}
\vphi\yz[K_{\xy}\scir\h\yz]=\vphi\yx[K_{\xy}]=\vphi\yx[\h\yx]=K_{\yx}=K_{\xz}\scir
K_{\yz}\per
\end{gather*}
These equations imply that the isomorphism $\vphih\xy$ induced by
$\vphi\xy$ on the quotient group $\G x/(\h\xy\scir\h\xz) $ must
coincide with $\vphi\xy$, the isomorphism $\vphih\yz$ induced by
$\vphi\yz$ on the quotient group $\G y/(K_{\xy}\scir\h\yz) $ must
coincide with $\vphi\yz$, which is the same as $\vphi\yx$ and
therefore also the same as $\vphi\xy\mo$, by frame condition (ii)
(which has been shown to hold by conditions (i) and (ii) of the
theorem) and Converse Theorem~\ref{T:convthm1},  and the isomorphism
$\vphih\xz$ induced by $\vphi\xz$ on the quotient group $\G
x/(\h\xy\scir\h\xz) $ must coincide with the identity automorphism
of $\G x/\h\xy $\per Consequently,
\[\vphih\xy\rp\vphih\yz=\vphi\xy\rp\vphi\yx=\vphi\xy\rp\vphi\xy\mo=\vphih\xz\per
\]

 Assume now, that the indices $x$, $y$, and $z$ are  all distinct
from one another. If $x<y<z$, then
\begin{gather*}
\vphi \xy[\h \xy\scir\h \xz]=K_{\xy}\scir\h\yz\comma\qquad \vphi
\yz[K_{\xy}\scir\h\yz]=K_{\xz}\scir K_{\yz}\comma\tag{2}\label{Eq:simpfr1}\\
\vphih\xy\rp\vphih\yz=\vphih\xz\comma\tag{3}\label{Eq:simpfr2}
\end{gather*}
by conditions (iii) and (iv) of the theorem, where  $\vphih\xy$ is
the isomorphism from $\G\wx/(\hh)$ to $\G\wy/(\kh)$ that is induced
by $\vphi\xy$, while $\vphih\yz$ is the   isomorphism from
$\G\wy/(\kh)$ to $\G\wz/(\kk)$ that is induced by $\vphi\yz$. It
follows from \eqref{Eq:simpfr2}  that $\vphih\xz$ must be the
isomorphism from $\G\wx/(\hh)$  to $\G\wz/(\kk)$ that is induced by
$\vphi\xz$ (this is not part of the assumption in condition (iii)).
Consequently,
\begin{equation*}\tag{4}\label{Eq:simpfr3}
\vphi \xz[\h \xy\scir\h \xz]=K_{\xz}\scir K_{\yz}.
\end{equation*}
The corresponding equations for  the  remaining pairs of mappings
follow readily from \eqref{Eq:simpfr1}, \eqref{Eq:simpfr3}, and
condition (ii). In more detail,
\begin{equation*}\tag{5}\label{Eq:simpfr10}
\vphi\yx=\vphi\xy\mo\comma\qquad\vphi\zy=\vphi\yz\mo\comma\qquad\vphi\zx=\vphi\xz\mo\comma
\end{equation*}
by  (the  already verified)  frame  condition (ii). In particular,
\begin{gather*}\tag{6}\label{Eq:simpfr4}
\h \zx = K_{\xz}\comma\quad K_{\zx}=\h\xz\comma \quad\h \zy =
K_{\yz}\comma\quad K_{\zy}=\h\yz,\\ \h \yx = K_{\xy}\comma\quad
K_{\yx}=\h\xy.
\end{gather*}
Apply $\vphi\zx$ to  both sides  of \eqref{Eq:simpfr3}, and use
\eqref{Eq:simpfr10}, to obtain
\[\vphi\zx[K_{\xz}\scir K_{\yz}]=\h\xy\scir\h\xz.
\]
With the help of \eqref{Eq:simpfr4}, rewrite this last equation as
\[\vphi\zx[\h\zx\scir\h\zy]=K_{\zx}\scir\h\xy\per
\] Use \eqref{Eq:simpfr4} and \eqref{Eq:simpfr1} repeatedly, together with the fact that the subgroups
involved are normal,  to obtain
\[\vphi\xy[K_{\zx}\scir\h\xy]=\vphi\xy[\h\xy\scir\h \xz]=K_{\xy}\scir\h\yz=K_{\xy}\scir K_{\zy}=K_{\zy}\scir K_{\xy}\per\]
In other words, condition (iii) holds with the variables $x$, $y$,
and $z$ replaced by $z$, $x$, and $y$ respectively.  The other cases
of the third frame condition are verified in a similar fashion.

Frame condition (iv) is  a simple  consequence of the preceding
observations, together with \eqref{Eq:simpfr2}  and
\eqref{Eq:simpfr10}. For example, compose both sides of
\eqref{Eq:simpfr2} on the left with $\vphih\yx$, and use
\eqref{Eq:simpfr10},
 to obtain
\begin{equation*}\tag{7}\label{Eq:simpfr5}
\vphih\yx\rp\vphih\xz=\vphih\yz\comma
\end{equation*}
and compose both sides of \eqref{Eq:simpfr2} on the right with
$\vphih\zy$, and use  \eqref{Eq:simpfr10},
  to obtain
\begin{equation*}\tag{8}\label{Eq:simpfr6}
\vphih\xz\rp\vphih\zy=\vphih\xy.
\end{equation*}  This argument shows that the two permuted versions
of \eqref{Eq:simpfr2}, the first obtained by transposing the first
two indices $x$ and $y$ of the triple $\trip xyz$, and the second by
transposing the last two indices $y$ and $z$ of the triple, are
valid in $\mc F$. All permutations of the triple $\trip xyz$ may be
obtained by composing these two transpositions in various ways. For
example, if we transpose the first two indices of
\eqref{Eq:simpfr2}, permuting $\trip xyz$ to $\trip yxz$ and
arriving at \eqref{Eq:simpfr5}, and then transpose the last two
indices of \eqref{Eq:simpfr5}, permuting $\trip yxz$ to $\trip yzx$,
we arrive at
\[\vphih \yz\rp\vphih\zx=\vphih\yz\per\] It follows that frame
condition (iv) is valid in $\mc F$.
\end{proof}

An examination of the preceding proof reveals that only condition
(i) is used to verify the second frame condition in the case when
$x=y$, and to verify the last two frame condition when $x=y$ or
$y=z$.  Also, only
  conditions (i) and (ii) are used to verify the last two frame
conditions when $x=z$.  The following corollary, which will be
needed in the construction of coset relation algebras, is a
consequence of this observation. In formulating it and the
succeeding two corollaries, we use the following simplified
notation: if $f$ is the $\alpha$th element in some fixed enumeration
of one of the  groups $\G x$ in a group pair, then we write $\hs\xx
f$ and $R_{\xx, f}$ for $\hs\xx\alpha$ and $R_{\xx,\alpha}$
respectively.
\begin{corollary}\label{C:compthma}  Let $\mc F$ be a group pair  satisfying condition
\textnormal{(i)} of Theorem~\textnormal{\ref{T:simpfr}}\per  The
following conditions hold for all $x$ in $I$ and all pairs $\pair
xy$ in $\mc E$\per
\begin{enumerate}
\item[(i)] $R_{\xx, f}\mo=R_{\xx, g}$ for $f$  in $\G x$ and $g=f\mo$\per
\item[(ii)]
$R_{\xx, f}\rp R_{\xy,\bt}=R_{\xy,\gm}$ for $f$ in $\G x$ and
$\hs\xy\gm= f \scir\hs\xy\bt$\per
\item[(iii)]
$R_{\xy,\alpha}\rp R_{\yy, g}=R_{\xy,\gm}$ for $g$  in $\G y$ and
$\quad\ks\xy\gm= \ks\xy\alpha\scir g$\per
 \item[(iv)]  If
$\mc F$ also satisfies condition \textnormal{(ii)} of
Theorem~\textnormal{\ref{T:simpfr}}\comma  then
\begin{align*}
            R_{\xy,\alpha}\rp R_{\yx,\bt}&=\tbigcup\{R_{\xx, f} : f \in
\hs\xy\alpha\scir\hs\xy\bt\}\\&=\{\pair g{g\scir f}:g\in  \G x\text{
and } f\in \hs\xy\alpha\scir\hs\xy\bt\}\per
           \end{align*}
\end{enumerate}
In particular\co each of these converses and  compositions is in the
set of frame relations\po\end{corollary}

\begin{proof} As an example, we prove (iv).  Write $z=x$ and use condition (i) from Theorem~\ref{T:simpfr} to see that
\begin{equation*}\tag{1}\label{Eq:compthma1}
  \h\xz=\h\xx=\{e_x\}
\end{equation*}
and that
\begin{equation*}\tag{2}\label{Eq:compthma2}
  \vphi\xz=\vphi\xx
\end{equation*} is the identity automorphism of $\G x/\{e_x\}$\per
Consequently,
\begin{equation*}\tag{3}\label{Eq:compthma3}
  \hs\xz f=\hs\xx f=\{f\}\qquad\text{and}\qquad R_{\xz, f}=R_{\xx, f}=\{\pair g{g\scir f}:g\in \G x\}
\end{equation*}
for every $f$ in $\G x$.  The additional assumption of condition
(ii) from Theorem~\ref{T:simpfr} implies that  frame condition (ii)
holds, by Theorem~\ref{T:simpfr}, and therefore
\begin{equation*}\tag{4}\label{Eq:compthma4}
  \vphi\yx=\vphi\xy\mo\per
\end{equation*} Invoke Convention~\ref{Co:convention} to write $\kai\yx=\kai\xy$ and
\begin{equation*}\tag{5}\label{Eq:compthma5}
  \hs\yx\g=\ks\xy\g\qquad\text{and}\qquad\ks\yx\g=\hs\xy\g
\end{equation*} for all indices $\g<\kai\xy$.  In particular, taking $\g=0$, we obtain
\begin{equation*}\tag{6}\label{Eq:compthma6}
\h\yx=K_{\xy}\qquad\text{and}\qquad K_{\yx}=\h\xy\per
\end{equation*}

Since
\begin{equation*}\tag{7}\label{Eq:compthma7}
 \h\xy\scir\h\xz=\h\xy\scir\{e_x\}=\h\xy\comma
\end{equation*}
by \eqref{Eq:compthma1}, the isomorphism $\vphih\xz$ induced by
$\vphi\xz$, which coincides with $\vphih\xx$, by
\eqref{Eq:compthma2}, is the identity automorphism of $\G
x/\h\xy$\comma by \eqref{Eq:compthma2}\per  Similarly,
\eqref{Eq:compthma7} implies that the isomorphism $\vphih\xy$
induced by $\vphi\xy$ coincides with $\vphi\xy$.  Finally,
\eqref{Eq:compthma4} and the preceding observation imply that the
isomorphism $\vphih\yz$ induced by $\vphi\yz$ coincides with
$\vphi\xy\mo$.  Combine these three observations to conclude that
\[\vphih\xy\rp\vphih\yz=\vphi\xy\rp\vphi\xy\mo\comma\] which then is the identity automorphism on $\G x/\h\xy$, and therefore
\begin{equation*}\tag{8}\label{Eq:compthma8}
 \vphih\xy\rp\vphih\yz=\vphih\xz\per
\end{equation*}

The assumption on $z$, and  \eqref{Eq:compthma5}, yield
\[\ks\xy\alpha\scir\hs\yz\bt=\ks\xy\alpha\scir\hs\yx\bt=\ks\xy\alpha\scir\ks\xy \bt\comma\] and therefore,
using also the isomorphism properties  of $\vphi\xy\mo$,
\begin{multline*}\tag{9}\label{Eq:compthma9}
 \vphi\xy\mo[\ks\xy\alpha\scir\hs\yz\bt]=\vphi\xy\mo(\ks\xy\alpha\scir\ks\xy\bt)\\
 =\vphi\xy\mo(\ks\xy\alpha)\scir\vphi\xy\mo(\ks\xy\bt)=\hs\xy\alpha\scir\hs\xy\bt\per
\end{multline*} Take $\alpha=\bt=0$ in \eqref{Eq:compthma9} to arrive at\begin{equation*}\tag{10}\label{Eq:compthma10}
\vphi\xy\mo[K_{\xy}\scir\h\yz]= \vphi\xy\mo[\ks\xy 0\scir\hs\yz
0]=\hs\xy 0\scir\hs\xy 0=\h\xy\scir\h\xy=\h\xy\per
\end{equation*}  Because $\h\xz$ coincides with the trivial subgroup $\{e_x\}$, by \eqref{Eq:compthma1}, it may be
concluded from \eqref{Eq:compthma10} that
\begin{equation*}\tag{11}\label{Eq:compthma11}\h\xz\seq  \vphi\xy\mo[K_{\xy}\scir\h\yz]\per
\end{equation*}   Together,   \eqref{Eq:compthma8} and \eqref{Eq:compthma11} show that condition (iv) in Composition
Theorem~\ref{T:compthm} is satisfied in the case under
consideration. Apply the implication from (iv) to (iii) in that
theorem, together with \eqref{Eq:compthma9}, \eqref{Eq:compthma3},
and the definition of $R_{\xx, f}$,  to conclude that
\begin{align*}R_{\xy,\alpha}\rp R_{\yz,\bt}&=R_{\xy,\alpha}\rp R_{\yx,\bt}\\
&=\tbigcup\{R_{\xx, f} : \hs\xx f\seq\vphi\xy\mo[\ks\xy\alpha\scir\hs\yz\bt]\}\\
&=\tbigcup\{R_{\xx, f} : f\in \hs\xy\alpha\scir\hs\xy\bt\}\\
&=\tbigcup\{\{\pair g{g\scir f}:g\in\G x\}:f\in \hs\xy\alpha\scir\hs\xy\bt\}\\
&=\tbigcup\{\pair g{g\scir f}:g\in\G x\text{ and }f\in
\hs\xy\alpha\scir\hs\xy\bt\}\per
 \end{align*}
\end{proof}

We return  to the question that was posed after the Converse
Theorem: can condition (i) in that theorem be replaced by the
condition that $R_{\xy, \alpha}\mo$ be in $A$ for some fixed
$\alpha$?

\begin{corollary}\label{C:convequiv}  If the set $A$ of frame relations contains the
identity relation  on the base set\comma  then for any pair $\pair x
y$ in $\mc E$\co the following conditions are equivalent\per
\begin{enumerate}
\item[(i)] $R_{\xy,\alpha}\mo$ is in $A$ for some $\alpha<\kai\xy$\per
\item[(ii)]$R_{\xy,\alpha}\mo$ is in $A$ for all $\alpha<\kai\xy$\per
\end{enumerate}
\end{corollary}
\begin{proof} The implication from (ii) to (i) is obvious. To establish the reverse
implication, assume that $R_{\xy,\x}\mo$ is in $A$ for some
$\x<\kai\xy$, and let $\alpha<\kai\xy$  be an arbitrary index.
Choose an element $f$ in $\G\wx$ such that
\begin{equation*}\tag{1}\label{Eq:convequiv1}
f\scir\hs\xy\x=\hs\xy\alpha\per
\end{equation*}
  The assumption on $A$ implies that the
group pair $\mc F$ satisfies condition (i) of
Theorem~\ref{T:simpfr}, by the Identity Theorem~\ref{T:identthm1}.
Apply Corollary~\ref{C:compthma}(ii) and \eqref{Eq:convequiv1} to
obtain
\begin{equation*}\tag{2}\label{Eq:convequiv4}
R_{\xx, f}\rp R_{\xy,\x}=R_{\xy,\alpha}\per
\end{equation*}   Form the converse
of both sides of \eqref{Eq:convequiv4}, and use the second
involution law for relational composition, to arrive at
\begin{equation*}\tag{3}\label{Eq:convequiv5}
R_{\xy,\x}\mo\rp R_{\xx, f}\mo=R_{\xy,\alpha}\mo.
\end{equation*}

Put
\begin{equation*}\tag{4}\label{Eq:convequiv2}
f\mo=g\per
\end{equation*}
Apply Corollary~\ref{C:compthma}(i) to \eqref{Eq:convequiv2} to
obtain
\begin{equation*}
R_{\xx, f}\mo = R_{\xx, g}.
\end{equation*}
Use this equation to rewrite equation  \eqref{Eq:convequiv5} in the
form
\begin{equation*}\tag{5}\label{Eq:convequiv6}
R_{\xy,\x}\mo\rp R_{\xx, g}=R_{\xy,\alpha}\mo.
\end{equation*}
The relation $R_{\xy,\x}\mo$ is in $A$, by assumption.
Corollary~\ref{C:compthma}(ii),(iii) and the dis\-trib\-u\-tiv\-ity
of relational composition  together imply that the set $A$ is closed
under the composition of its elements with relations of the form
$R_{\xx, g}$. Consequently, the composition on the left side of
\eqref{Eq:convequiv6} is in $A$. Use  \eqref{Eq:convequiv6} to
conclude that $R_{\xy,\alpha}\mo$ is in $A$.
\end{proof}

Turn next to the question that was posed after the Composition
Theorem: can condition (i) in that theorem be replaced by the
condition that $R_{\xy, \alpha} \rp R_{\yz, \bt}$ be in $A$ for some
fixed $\alpha$ and $\bt$?
\begin{corollary}\label{C:compequiv}  If the set $A$ of frame relations contains the
identity relation\comma then for any pairs $\pair x y$ and $\pair y
z$ in $\mc E$\co the following conditions are equivalent\per
\begin{enumerate}
\item[(i)] $R_{\xy,\alpha}\rp R_{\yz,\bt}$ is in $A$ for some $\alpha<\kai\xy$ and
some $\bt<\kai\yz$\per
\item[(ii)]$R_{\xy,\alpha}\rp R_{\yz,\bt}$ is in $A$ for all $\alpha<\kai\xy$ and
all $\bt<\kai\yz$\per
\end{enumerate}
\end{corollary}
\begin{proof}  The
 implication from (ii)  to (i) is  obvious.  To establish the
reverse implication,  use an argument  similar to the one in the
preceding proof. Assume that $R_{\xy,\x}\rp R_{\yz,\eta}$ is in $A$
for some $\x<\kai\xy$ and $\eta<\kai\yx$. Let $\alpha <\kai\xy$ and
$\bt<\kai\yz$ be arbitrary, and choose elements $f$ in $\G x$ and
$g$ in $\G z$
 so that
\begin{align*}
f\scir\hs\xy\x=\hs\xy\alpha\qquad&\text{and}\qquad\ks\yz\eta\scir
 g=\ks\yz\bt\per\tag{1}\label{Eq:compequiv1}\\ \intertext{The assumption
on $A$ implies that the group pair $\mc F$ satisfies condition (i)
of Theorem~\ref{T:simpfr}, by the Identity
Theorem~\ref{T:identthm1}. Apply Corollary~\ref{C:compthma}(ii),(ii)
and \eqref{Eq:compequiv1} to obtain}
 R_{\xx, f}\rp R_{\xy,\x}=R_{\xy,\alpha}\qquad&\text{and}\qquad
R_{\yz,\eta}\rp R_{\zz, g}=R_{\yz,\bt}\per\end{align*} These
equations and the associative law for relational composition lead
immediately to
\begin{gather*}\tag{2}\label{Eq:compequiv2}R_{\xy,\alpha}\rp
R_{\yz,\bt} = (R_{\xx, f}\rp R_{\xy,\x})\rp(R_{\yz,\eta}\rp R_{\zz,
g}) =\\
R_{\xx, f}\rp(R_{\xy,\x}\rp R_{\yz,\eta})\rp R_{\zz, g}\per
\end{gather*}
The relation $R_{\xy,\x}\rp R_{\yz,\eta}$ is in $A$, by assumption.
Corollary~\ref{C:compthma} and the dis\-trib\-u\-tiv\-ity of
relational composition over unions together imply that the set $A$
is closed under the composition of its elements with relations of
the form $R_{\xx, f}$ and $R_{\zz, g}$. Consequently, the
composition on the right side of \eqref{Eq:compequiv2} is in $A$.
Use \eqref{Eq:compequiv2} to conclude that $R_{\xy,\alpha}\rp
R_{\yz,\bt} $ is in $A$.
 \end{proof}

\section{Examples}

 The easiest group frame to construct involves a kind of ``power" of
a quotient group. Fix a group $M$ and  a normal subgroup $N$. For
each element $x$ in a given index set $I$, let $\grp x$ be an
isomorphic copy of $M$ (chosen so that distinct copies are pairwise
disjoint) and $\psi_x$ an isomorphism from the quotient group $M/N$
to the corresponding quotient group of $\grp x$\per Take the mapping
$\ho x y$ to be the natural isomorphism between the quotient groups
of $\grp x$ and of $\grp y$\comma defined by
\[\ho x y =\psi_x\mo\rp\psi_y
\]
 for distinct $x$ and $y$ in $I$. These mappings are all
 isomorphisms between  copies of the single quotient group $M/N$.
 Take $\ho x x$ to be  the  identity automorphism of $\grp x/\{\cs
 ex\}$, as required by the definition of a  frame.

The resulting  pair $\mc F=\pair G \varphi$ is readily seen to be a
group frame, and the corresponding  group relation algebra $\cra G
F$ is  a measurable set relation algebra. If we take the indices
$\alpha$ of the atomic relations in $\cra  G F$ to be the
corresponding cosets of $M/N$, then there are especially simple
formulas for computing the converse of an atomic relation and the
composition of two atomic relations in $\cra G F$:
\[R_{xy,\alpha}\mo =
\newrr y x {\alpha\mo}\qquad\text{and}\qquad \newrr x y
\alpha\relprod\newrr y z \beta = \newrr x z {\alpha\scir\beta}
\] when $x$, $y$, and $z$ are distinct elements of $I$.  (Here
$\alpha\mo$ denotes the inverse of the coset $\alpha$, and
$\alpha\scir\beta$ the product of the  cosets $\alpha$ and $\beta$,
in $M/N$.)

If $N$ is the trivial (one-element) subgroup of $M$, then each
atomic relation $\newrr x y \alpha$ is a function and in fact a
bijection from $\grp x$ to $\grp y$\per  In this case, $\cra G F$ is
an example of an atomic relation algebra with functional atoms. At
the other extreme, if $N$ coincides with $M$, then there is only one
atomic relation, namely
\[\newrr x y 0=\grp x\times \grp y\comma\]
for each pair of distinct indices $x,y$ in $I$. In general, if  the
normal subgroup $N$ has order $\lambda$ and index $\kappa$ in $M$
(that is to say, if $N$ contains $\lambda$ elements and has $\kappa$
cosets in $M$), then there will be $\kappa$ distinct atomic
relations of the form $\newrr x y \alpha$\comma and each of them
will be the union of $\lambda$ pairwise disjoint bijections from
$\grp x$ to $\grp y$.

In the general case of the power construction, $\mce$ is allowed to
be an arbitrary equivalence relation on $I$.  Moreover, the group
$M$ and normal subgroup $N$ are fixed for a given equivalence class
of $\mce$, but different equivalence classes may use different
groups and normal subgroups.

The most trivial case of the power construction is when the fixed
group $M$ is the   one-element group.  In this case, $\cra  G F$ is
just the full set relation algebra with base set and unit
\[U=\tbigcup \{\grp x:x\in I\}\qquad\text{and}\qquad
E=\tbigcup\{\cs Gx\times\cs Gy:\pair xy\in\mc E \}\] respectively.
Moreover, every full set relation algebra on an equivalence relation
can be obtained as a group relation algebra in this fashion, using
arbitrary equivalence relations $\mc E$ on $I$. The construction of
full set relation algebras may therefore be viewed as the most
trivial case of the construction of full group relation algebras,
namely the case when all the groups have order one. This class may
be characterized abstractly, up to isomorphisms, as the class of
complete and atomic singleton-dense relation algebras.

It follows from this observation that the class of algebras
embeddable into full group relation algebras coincides with the
class of representable relation algebras.  In particular, the class
is equationally axiomatizable, by the results of Tarski\,\cite{t55}.
However, the description of representable relation algebras in terms
of group relation algebras seems much more advantageous, because the
class of full group relation algebras is substantially more varied
and interesting than the class of full set relation algebras.

A second example  of the  group relation algebra construction that
is easy to describe is the one in which all of the groups are
cyclic. Suppose $G=\langle\grp x:x\in I\,\rangle$ is a family of
(pairwise disjoint) cyclic groups and $\mce$ an equivalence relation
on $I$. To avoid unnecessary complications in notation, we  consider
here only the case when the  groups are finite.  Fix a generator
$g_x$ of each group $\grp x$.  Let $\langle \kap x y:\pair x y \in
\mce\,\rangle$ be a system  of positive integers satisfying the
following  conditions for all appropriate pairs in $\mce$\per
\begin{enumerate}
\item[(i)] $ \kap x y$ is a common divisor of the orders of $\grp x$
and $\grp y$\per
\item[(ii)] $ \kap x x$ is equal to the order of $\grp x$\per
\item[(iii)] $ \kap y x = \kap x y$\per
\item[(iv)] $\gcd( \kap x y,\kap y z) =  \gcd( \kap x y,\kap x
z)=\gcd( \kap x z,\kap y z)$\per
\end{enumerate}

  Condition (i) ensures that there are (uniquely determined)
subgroups $\newhl x y$ and $\newhr x y$ of index $\kap x y$ in $\grp
x$ and $\grp y$ respectively.    The  quotient groups $\newgl x y$
and $\newgr x y$ are therefore isomorphic, and in fact there is a
uniquely determined isomorphism $\ho x y$ between them that maps the
generator $g_x/\newhl x y$ of the first quotient to the generator
$g_y/\newhr x y$ of the second. Conditions (ii) and (iii), and the
definition of the quotient isomorphisms, ensure that  frame
conditions (i) and (ii) are satisfied. The complex product $\newhl x
y\scir\newhl x z$ is a subgroup of $\grp x$ of index $d = \gcd (\kap
x y, \kap x z)$. Condition (iv) says that the complex products
$\newhr x y\scir\newhl y z$ and $\newhr x z\scir\newhr y z$ also
have index $d$. This, together with the definition of the quotient
isomorphisms, ensures that frame conditions (iii) and (iv)  are
satisfied. It follows that the  pair $\mc F =(G,\varphi)$ is a group
frame. This construction using cyclic groups is due jointly to
Hajnal Andr\'eka and the author.

If every group in  $\mc F$ has order one or two, then the group
relation algebra $\cra G F$ is an example of a pair-dense relation
algebra in the sense of Maddux \cite{ma91}.  When $\kap x y = 2$,
there are exactly two relations: $\newrr x y 0$ and  $\newrr x y
1$\per Each of them is a function, and  in fact a bijection from
$\grp x$ to $\grp y$, with exactly two pairs.  When $\kap x y = 1$,
there is only the one relation  $\newrr x y 0=\grp x\times \grp
y$\per It contains either four pairs, two pairs, or one pair,
according to whether both groups $\grp x$ and $\grp y$ have order
two, exactly one of these  groups has order two and the other order
one, or both groups have order one.  The class of such group
relation algebras may be characterized abstractly, up to
isomorphisms, as the class of complete and atomic pair-dense
relation algebras.

\section{A decomposition theorem}

The  isomorphism index set $\mc E$ of a group frame $\mc F=\pair
G\vp$ is  an equivalence relation on the group index set $I$, and
the unit \[E=\tbigcup\{\G\wx\times\G\wy:\pair xy\in \mc E \}\] of
the corresponding full group relation algebra $\cra G F$ is an
equivalence relation on the base set $U=\tbigcup_{x\in I}\cs Gx$.
Call a group frame \textit{simple} if the group index set $I$  is
not empty, and the isomorphism  index set $\mc E$
 is the universal relation on $I$. It turns out that the frame $\mc F$ is simple
 if and only if the algebra $\cra GF$ is simple in the algebraic sense
of the word, namely, it has more than one element, and every
non-constant homomorphism on it must be injective; or, equivalently,
it has exactly two ideals, the trivial ideal and the improper ideal.
\begin{theorem}\label{T:simplegra}
  Let $\mc F$ be a group frame\per The group relation algebra $\cra GF$ is simple if and only if   $\mc
  F$is simple\per
\end{theorem}
\begin{proof}
  Suppose first that the frame $\mc F$ is simple.  The isomorphism index set $\mc E$ is then
  the universal relation on the index set $I$, and consequently the unit $E$
  of $\cra GF$ is the universal relation $U\times U$ on the base
  set $U$.  Moreover, the base set $U$ is not empty, because the index set $I$ is not empty,
  and the groups indexed by $I$ are not empty.  The algebra $\full E$ therefore consists
  of all binary relations on
  a non-empty base set. Such set relation algebras are  well known to be simple.
  Moreover,   $\cra GF$ is a
  subalgebra of  $\full E$, by Group Frame Theorem~\ref{T:closed}.  It is well known that subalgebras of simple
  relation algebras are simple, so $\cra GF $
  must also be
  simple.

  We postpone the proof of the reverse implication, that
  simplicity of $\cra GF$ implies that of $\mc F$,
 until after the next theorem.
\end{proof}

 It turns out that
every full group relation algebra can  be decomposed into the direct
product of   simple, full group relation  algebras. Here is a sketch
of the main ideas. The details are left  to the reader. Given an
arbitrary group frame
\[\mc F=(\langle \G x:x\in
I\,\rangle\smbcomma\langle\vph_\xy:\pair xy\in \mc E\rangle)\comma\]
consider an equivalence class $J$ of the isomorphism index set $\mc
E$. The universal relation $J\times J$ on $J$ is a subrelation of
$\mc E$, and in fact it is a maximal connected component of $\mc E$
in the graph-theoretic sense of the word. The
  \textit{restriction} of $\mc F$ to $J$  is defined to be the
group pair
\[\mc F_J=(\langle \G x:x\in
J\,\rangle\smbcomma\langle\vph_\xy:\pair xy\in J\times J\rangle).\]
Each such restriction of $\mc F$ to an equivalence class of the
index set $\mc E$ inherits the frame properties of $\mc F$ and is
therefore a simple group frame. Call such restrictions the
\textit{components} of $\mc F$.  It is not difficult to check that
every frame is the disjoint union of its components in the
  sense that the group system and the isomorphism system of $\mc
F$ are obtained by respectively combining the group systems and the
isomorphism systems  of the components of $\mc F$.

Each  component $\cs{\mc F}J$ gives rise to a full group relation
algebra $\cras G {\cs {\mc F}J }$ that is  simple and is in fact a
subalgebra of the full set relation algebra with   base set and unit
\[\cs UJ=\tbigcup_{x\in J}\cs Gx\qquad\text{and}\qquad\cs EJ=\cs
UJ\times\cs UJ\] respectively. The group relation algebra $\cra GF$
is isomorphic to the direct product of the simple group relation
algebras $\cras G{\cs {\mc F}J}$ constructed from the components of
$\mc F$ (so $J$ varies over the equivalence classes of $\mc E$). In
fact, if internal direct products are used instead of Cartesian
direct products, then $\cra G F$ is actually equal to the internal
direct product of the full group relation algebras constructed from
its component frames.
\begin{theorem}[Decomposition Theorem]\label{T:cosdecomp} Every  full group
relation algebra is isomorphic to a direct product of full group
relation algebras on simple frames\per
\end{theorem}

Return now to the proof of the reverse implication in
Theorem~\ref{T:simplegra}. Assume that the    frame $\mc F$ is not
simple. If the group index set $I$ is empty, then the base set $U$
is also empty, and in this case $\cra G F$ is a one-element relation
algebra with the empty relation as its only element. In particular,
$\cra GF$ is not simple. On the other hand, if the group index set
$I$ is not empty, then the isomorphism index set $\mc E$ has at
least two equivalence classes, by the definition of a simple frame.
The group relation algebra $\cra G F$ is isomorphic to the direct
product of the group relation algebras on the component  frames of
$\mc F$, by Decomposition Theorem~\ref{T:cosdecomp}, and there are
at least two such components. Each of these components is a simple
frame, so the corresponding group relation algebra is simple, by the
first part of the proof of Theorem~\ref{T:simplegra}. It follows
that $\cra GF$ is isomorphic to a direct product of at least two
simple relation algebras, so $\cra GF$ cannot be simple. For
example, the projection of $\cra G F$ onto one of the factor
algebras is a non-constant homomorphism that  is not injective.

\section{Summary}

The present paper generalizes the notion of pair density from
Maddux\,\cite{ma91} by introducing  the notion of a measurable
relation algebra.  A large class of examples of such algebras has
been constructed, namely the class of full group relation algebras.
Unfortunately, the class is not large enough to represent all
measurable relation algebras: there exist measurable relation
algebras that are not essentially isomorphic to (full) group
relation algebras, and in fact that are not representable as set
relation algebras at all. The next paper  in this series, \cite{ag},
greatly extends the class of examples of measurable relation
algebras by adding one more ingredient to the mix, namely systems of
cosets that are used to modify the operation of relative
multiplication. In the group relation algebras constructed in the
present paper, the operation of relative multiplication is just
relational composition, but in the coset relation algebras to be
constructed in the next paper, the operation of relative
multiplication is ``shifted" by coset multiplication, so that in
general it no longer coincides with composition.  On the one hand,
this shifting leads to examples of measurable relation algebras that
are not representable as  set relation algebras, see \cite[Theorem
5.2]{ag}. On the other hand, the class of coset relation algebras
constructed from systems of group pairs and shifting cosets really
is broad enough to include all measurable relation algebras. The
task of the third paper in the series, \cite{ga},  is to prove this
assertion, namely that every measurable relation algebra is
essentially isomorphic to a coset relation algebra, see
\cite[Theorem 7.2]{ga}.


\subsection*{Acknowledgment}

The author is very much indebted to Dr.\,Hajnal Andr\'eka, of the
Alfr\'ed R\'enyi Mathematical Institute in Budapest, for carefully
reading a draft of this paper and making many extremely helpful
suggestions.


\end{document}